\documentclass[letterpaper,11pt]{article}

\usepackage{ucs}
\usepackage[utf8x]{inputenc}
\usepackage{graphicx}
\usepackage{amsfonts}
\usepackage{dsfont}
\usepackage{amssymb}
\usepackage{amsmath}
\usepackage{amsthm}
\usepackage{enumerate}
\usepackage{stmaryrd}
\usepackage{fullpage}
\usepackage{ifthen}
\usepackage{subfigure}
\usepackage{epic}
\usepackage{authblk}
\usepackage{textcomp}
\usepackage[small]{caption}

\usepackage[hypertexnames=false,colorlinks=true,linkcolor=blue,citecolor=blue]{hyperref}
\usepackage[numbers,comma,square,sort&compress]{natbib}
\usepackage[letterpaper,text={7in,9in},centering]{geometry}

\usepackage{color}
\usepackage{titlesec}
\graphicspath{{eps/}{pdf/}}

\newcommand{\be}{\begin{equation}}
\newcommand{\ee}{\end{equation}}
\newcommand{\bqs}{\begin{equation*}}
\newcommand{\eqs}{\end{equation*}}

\newcommand{\gL}{\mathrm{L}}
\newcommand{\gA}{\mathrm{A}}

\newcommand{\gM}{\mathrm{M}} 
\newcommand{\gI}{\mathrm{I}}

\renewcommand{\v}{\mathbf{v}}
\renewcommand{\u}{\mathbf{u}}
\newcommand{\ve}{\mathbf{e}} 
\newcommand{\w}{\mathbf{w}}
\newcommand{\bomega}{\boldsymbol{\omega}}

\newcommand{\mbi}{\mathrm{i}}

\newcommand{\rmd}{\mathrm{d}}

\renewcommand{\L}{\mathcal{L}(W)}
\renewcommand{\O}{\mathcal{O}}

\newcommand{\R}{\mathbb{R}}

\newcommand{\e}{\varepsilon}

\numberwithin{equation}{section}

  {\begin{trivlist}\item[]{\bf Hypothesis #1 }\em}{\end{trivlist}}

\newcommand\fH[1]{\sbox0{#1}\dimen0=\ht0 \advance\dimen0 -1ex
  \sbox2{\'{}}\sbox2{\raise\dimen0\box2}%
  {\ooalign{\hidewidth\kern.1em\copy2\kern-.5\wd2\box2\hidewidth\cr\box0\crcr}}}

\theoremstyle{plain}
\newtheorem{theorem}{Theorem}[section]
\newtheorem{proposition}[theorem]{Proposition}
\newtheorem{lemma}[theorem]{Lemma}
\newtheorem{corollary}[theorem]{Corollary}

\newtheorem{rmk}[theorem]{Remark}
\newtheorem{defn}[theorem]{Definition}
\newtheorem{example}[theorem]{Example}

 {\begin{trivlist}\item[]\textbf{Proof#1 }}%
 {\hspace*{\fill}$\rule{0.3\baselineskip}{0.35\baselineskip}$\end{trivlist}}

\title{Pattern Formation in Random Networks Using Graphons  }

\author[1,2]{Jason Bramburger}
\author[1,3]{Matt Holzer}
\affil[1]{\small Department of Mathematical Sciences, George Mason University, Fairfax, VA, USA }
\affil[2]{\small Department of Mathematics and Statistics, Concordia University, Montr\'eal, QC, Canada}
\affil[3]{\small Center for Mathematics and Artificial Intelligence (CMAI), George Mason University, Fairfax, VA, USA}

\date{}

\begin{document}
\maketitle

\begin{abstract}
We study Turing bifurcations on one-dimensional random ring networks where the probability of a connection between two nodes depends on the distance between them.  Our approach uses the theory of graphons to approximate the graph Laplacian in the limit as the number of nodes tends to infinity by a nonlocal operator -- the graphon Laplacian.  For the ring networks considered here, we employ center manifold theory to characterize Turing bifurcations in the continuum limit in a manner similar to the classical partial differential equation case and classify these bifurcations as sub/super/trans-critical. We derive estimates that relate the eigenvalues and eigenvectors of the finite graph Laplacian to those of the graphon Laplacian.  We are then able to show that, for a sufficiently large realization of the network, with high probability the bifurcations that occur in the finite graph are well approximated by those in the graphon limit. The number of nodes required depends on the spectral gap between the critical eigenvalue and the remaining ones, with the smaller this gap the more nodes that are required to guarantee that the graphon and graph bifurcations are similar.  We demonstrate that if this condition is not satisfied then the bifurcations that occur in the finite network can differ significantly from those in the graphon limit.

\end{abstract}



\section{Introduction }

Diffusive instabilities occur when a stable homogeneous state is destabilized in the presence of diffusion.  These instabilities form a common pathway to the spontaneous self-organization of patterned states.  Alan Turing discovered these bifurcations in work aimed at explaining the emergence of patterns during morphogenesis \cite{turing}, and today these instabilities are often referred to as Turing bifurcations.  Turing bifurcations have been shown to be a dominant mechanism for pattern formation in a number of applications including biology \cite{gierer72,kondo10,othmer71}, chemistry \cite{castets90,ouyang91}, neuroscience \cite{bressloff96,ermentrout79} and ecology \cite{klausmeier99,mimura78} and similar mechanisms are known to drive pattern formation in fluids \cite{cross,newell69} and bacterial aggregation \cite{keller70}.

In this work we will study Turing bifurcations occurring in complex networks described by random graphs with a particular focus on ring networks where the likelihood of connection between nodes is dependent on the distance between them.  An obvious challenge in the study of bifurcations on random graphs is deriving results that hold independently of the realization of the network with high probability.  To accomplish this we will appeal to the theory of graph limits utilizing graphons; see \cite{lovasz12}.  For a fixed random graph model, in the limit as the number of nodes goes to infinity, graphons approximate the adjacency matrix (or in our case the graph Laplacian) of the finite graph by a non-local operator.  The integral  kernel of this operator is called the graphon, which derives from {\em graph function}.  For the ring networks studied here, the graphon exactly describes the probability of connections between nodes as a function of the distance between them.  The question then becomes, essentially, how well we can predict the character of Turing bifurcations in the random graph based only upon some mean-field approximation of the average likelihood of a connection between nodes based upon their distance.  

A number of previous studies have focused on Turing bifurcations in complex networks.  For cellular networks, Turing instabilities were studied in \cite{othmer71} with a particular focus on several classes of regular networks.  In \cite{nakao10}, instability criteria are derived in analogy with the partial differential equation (PDE) theory where the onset of instability can be expressed as a function of the Laplacian eigenvalues.  In \cite{wolfrum12} it is demonstrated that for scale-free networks these bifurcations are typically transcritical and lead to the appearance of highly localized patterns dominated by a small number of differentiated nodes.  Snaking properties of these localized patterns were subsequently analyzed in \cite{mccullen16}.  Further studies of this phenomena include \cite{ide16,moore05,muolo19}, to only name a few.  We also mention recent work on Turing patterns in zebrafish resulting from nonlocal, distance-dependent interactions between cells \cite{kondo17,volkening18}. In a somewhat different direction, the study of localized pattern formation in the discrete spatial setting has been studied previously on infinite chains \cite{Bramburger1D,BramburgerIsola}, square lattices \cite{BramburgerSquare}, and rings \cite{TianRing}.  We also mention in passing the prevalence of non-local, distance dependent coupling in neural field models; see for example \cite{amari77,bressloff96,hopfield84}.

Our approach using graphons in this manuscript is motivated by recent work where graphons are used to approximate the interaction graph in a coupled oscillator problem and to derive bifurcations to synchrony in this mean-field limit, as was done in \cite{chiba18,chiba19}.   Other utilizations of graphons in the study of dynamical systems on random graphs include \cite{k17,medvedev14,medvedev14a} for nonlinear heat equations and \cite{kuehn19} for power networks.  

The theory of graph limits via graphons is the primary analytical tool that we use in this work (in addition to center manifold theory).  Graphons were introduced in \cite{Borgs,Borgs2} and the simplest way to visualize one is as a pixel plot of the adjacency matrix for a graph.  The idea is that, for some classes of random graphs, the sequence of realizations of the graph adjacency matrix will converge to a non-local operator as the number of vertices tends to infinity with many features of the graphs being preserved in the limit.  For our purposes here, we will be interested in how well the eigenvalues and eigenvectors of the random graph Laplacian will be approximated by the eigenvalues and eigenfunctions of the graphon Laplacian.  Informally speaking, in scenarios where this approximation is sufficiently accurate we will be able to show that the bifurcations occurring on the graph closely mimic the bifurcations occurring for the graphon.  More precisely, we will use spectral convergence results presented in \cite{Janson,lovasz12}, modified to the case of the graph Laplacian using the approach of \cite{Oliveira}, followed by an application of a version of the Davis-Kahan Theorem \cite{DK,AltDK} to control the difference between the graph Laplacian eigenvectors and the graphon eigenfunctions.    We refer the interested reader to \cite{Borgs,Borgs2,Glasscock,Janson,lovasz12,lovasz06} and references therein for more complete introductions to the theory of graph limits and graphons.  

Let us now set the stage for the problem considered herein. Consider a network described as a connected, un-weighted, un-directed graph $G=(V,E)$ with $|V|=N$ vertices.   In this paper, we will study the following system of ordinary differential equations, known as the Swift-Hohenberg equation, and given by
\be \label{eq:main} 
	\frac{d\u}{dt}=-(\gL-\kappa\gI)^2\u+\e \u +r \u\circ \u -b \u \circ \u\circ \u, \qquad \u = \u(t) \in \R^N
\ee
Here $\kappa<0$, $r\in\mathbb{R}$ and $b>0$ are parameters and $\u\circ \u$ is the Hadamard, or component-wise, product of vectors.  The matrix $\mathrm{L}$ is the graph Laplacian associated to the graph $G$ which describes interactions between connected elements on the graph. Throughout this work we will exclusively work with the combinatorial Laplacian defined by $\mathrm{L}=\mathrm{A}-\mathrm{D}$, where $\mathrm{A}$ is the {\em adjacency matrix} and $\mathrm{D}$ is the diagonal matrix composed of row sums of $\mathrm{A}$, termed the {\em degree matrix}.  {\color{black} In the PDE setting where the graph Laplacian $\mathrm{L}$ is replaced by the Laplacian differential operator, the Swift-Hohenberg equation is a model equation for systems undergoing a Turing bifurcation \cite{cross}, thus making it an optimal phenomenological model for our study of such bifurcations in the discrete spatial setting of graph networks.   The parameter $\kappa<0$ tunes which eigenvalues of the graph Laplacian $\mathrm{L}$ are unstable for (\ref{eq:main}), as we observe next.  }

{\color{black}The criterion for diffusive instability in \eqref{eq:main} is analogous to that of the PDE case, as described in \cite{nakao10}.  The graph Laplacian $\gL$ has eigenvalues $\lambda_k\leq 0$ and normalized eigenvectors $\v_k$, and so letting $\gM=-(\gL-\kappa\gI)^2$ means that the $\v_k$ remain eigenvectors of $\gM$, now with eigenvalues $\ell_k= -(\lambda_k-\kappa)^2$, expanded as $\ell_k = -\lambda_k^2+2\kappa\lambda_k-\kappa^2$. The importance of this is that the linearization of \eqref{eq:main} about the trivial homogeneous state $\u = 0 \in \mathbb{R}^N$ has eigenvalues $\ell_k + \varepsilon$, and so taking $\kappa = \lambda_k$ for some $k$ causes this state to lose stability as $\varepsilon$ crosses zero. This instability leads to a bifurcation in \eqref{eq:main} for which a branch of heterogeneous steady-states resembling $\v_k$ for small $\varepsilon > 0$ emerge from the bifurcation point at $\varepsilon = 0$. The emergence of such a heterogenous state is a type of spontaneous pattern formation, while the bifurcation that this results from is a Turing bifurcation. In this manuscript we are interested in understanding what types of patterns can form via such Turing bifurcations, particularly in the case of random graph networks, and how their bifurcation curves depend on the network topology encoded in $\mathrm{L}$ in the Swift--Hohenberg equation \eqref{eq:main}.}

Since \eqref{eq:main} is a system of nonlinear ODEs we may apply center manifold reduction techniques to characterize the bifurcation that occurs.  When $\gL$ is the graph Laplacian for an arbitrary complex network it is typically the case that the bifurcation occurring at $\e=0$ has co-dimension one and, owing to the preservation of the zero solution, the reduced equation on the {\color{black} one-dimensional} center manifold then takes the form 
\be\label{CenterIntro}
	 \frac{dw_1}{dt}=\e w_1+a_2 w_1^2+a_3w_1^3 +\mathcal{O}( w_1^4, w_1^2\e), 
\ee
for some constants $a_2$ and $a_3$.  The bifurcation can then be classified as transcritical if $a_2\neq 0$ and as a sub (super)-critical pitchfork if $a_2=0$ and $a_3>0$  ($a_3<0$). If $\v_k$ is the marginally stable eigenvector then the coefficient $a_2=r \v_k^T (\v_k\circ \v_k)$ measures interactions of the bifurcating mode with itself under the quadratic coupling term. {\color{black}Therefore, we see that determining the precise nature of the eigenvectors of the graph Laplacian $\gL$ is critical to not only understanding what pattern is emerging from the Turing bifurcation, but also to understanding the nature of the bifurcation itself. }

We emphasize that for complex networks we expect $a_2\neq 0$ and therefore the resulting bifurcation is generically transcritical. This is exactly what the analysis for scale-free networks has revealed \cite{wolfrum12}.  However, an important distinction exists for ring networks studied herein.  We will show that for sufficiently large realizations of these networks the isolated eigenvalues possess eigenvectors that are approximated by discrete Fourier basis vectors.  Such Fourier basis vectors have the property that they do not exhibit any self-interaction under the quadratic coupling in \eqref{eq:main} -- hence $a_2=0$ for the mean-field limit. This will be one of the main analytical results of this paper. Informally, if our random graph Laplacian has a sufficiently isolated eigenvalue and the number of nodes is sufficiently large, then the corresponding eigenvector will, {\color{black} with high probability}, be well approximated by a discrete Fourier basis vector and therefore $a_2\approx 0$. From \eqref{CenterIntro}, the accompanying Turing bifurcation on the random graph is transcritical with a saddle-node bifurcation nearby, coming from perturbing the mean-field pitchfork bifurcation. Theorem~\ref{thm:mainrandom} below provides a precise statement of this result.  

%
%
%

The remainder of the paper is organized as follows.  In Section~\ref{sec:graphon} we review the notion of a graphon and derive our main approximation results that relate the eigenvalues and eigenvectors of the random graph Laplacian to the spectrum and eigenfunctions of the graphon Laplacian operator.  In Section~\ref{sec:TuringonW}, we apply infinite dimensional center manifold theory to classify bifurcations for the non-local graphon Laplacian operator.  In Section~\ref{sec:Turinggraph}, we apply the results from Section~\ref{sec:graphon} to derive analogous bifurcation results on for the random graph Laplacian for realizations of the random network with a sufficiently large number of nodes.  In Section~\ref{sec:numerics}, we study some examples numerically both to lend credence to our main results and to demonstrate what happens when the bifurcating eigenvalue is not isolated or an insufficiently large realization of the network is taken. We conclude with a discussion of our results and areas for future exploration in Section~\ref{sec:Discussion}.


\section{Graphons as Graph Limits} \label{sec:graphon}

As stated in the introduction, our goal is to investigate bifurcations from the trivial state of the spatially discrete Swift--Hohenberg equation \eqref{eq:main}. In particular, we are interested in large graphs, $|V| = N \gg 1$, and therefore we turn to understanding their formal limiting objects as $N \to \infty$. In the following subsection we will review the relevant definitions on graph limits, termed graphons, and then in the subsections that follow we present a number of auxiliary results that describe the limiting behaviour of the spectrum of \eqref{eq:main} linearized about $\u = 0$ in the large $N$ limit. We take the perspective that we start with the limiting graphon and show how it can be used to define both deterministic and random graphs, which are treated separately in Section~\ref{subsec:Deterministic} and Section~\ref{subsec:Random}, respectively. Our review of graphons is in no way meant to be complete and therefore we direct the interested reader to \cite{Borgs,Borgs2,Glasscock,Janson,lovasz12,lovasz06} for a more thorough introduction and treatment of them.


\subsection{Preliminaries}

At the most general level a {\em graphon} is a symmetric Lebesgue-measurable function $W:[0,1]^2 \to [0,1]$. {\color{black} Our results will require that $W$ is continuous almost-everywhere in $[0,1]^2$, however the exposition in this subsection will not yet require this property.} A simple consequence of the fact that $0 \leq W \leq 1$ is that $W$ has finite $L^p := L^p([0,1]^2)$ norm,
\be
 ||W||_p := \left(\int_{[0,1]^2} |W(x,y)|^p \rmd x \rmd y\right)^{\frac{1}{p}}\leq 1, 
\ee 
for all $p \in [1,\infty)$ and similarly $\|W\|_\infty \leq 1$. The term graphon derives from ``graph function'' and they have the property that they approximate the pixel plot of an adjacency matrix in the limit as the number of vertices tends to infinity. Hence, this formalism can be used to obtain a non-local version of the graph Laplacian, denoted $\mathcal{L}(W)$ and termed a {\em graphon Laplacian} throughout, acting on functions $f:[0,1] \to \mathbb{C}$ by
\be\label{GraphonLap}
	\mathcal{L}(W) f(x) = \int_0^1 W(x,y) \left(f(y)-f(x)\right)\rmd y, \quad \forall x \in [0,1]. 
\ee
With the definition of a graphon Laplacian comes the notion of the degree of a vertex $x \in [0,1]$, given by
\begin{equation}
	\mathrm{Deg}(W)(x) := \int_0^1 W(x,y)\rmd y, 
\end{equation}
which can be interpreted as the formal limit of the degree matrix associated to a graph Laplacian.

Such graphon Laplacians can be seen as operators acting between $L^p$ spaces, and so we will consider the norm for linear operators mapping $L^p \to L^q$, given by
\be\label{OperatorNorm}
	\|\mathcal{L}(W)\|_{p\to q} := \sup_{\|f\|_p = 1} \|\mathcal{L}(W) f\|_q.
\ee 
In what follows we will primarily be concerned with the case $p = q = 2$ since $L^2$ is a Hilbert space, thus allowing one to project onto eigenspaces. Moreover, notice that the symmetry $W(x,y) = W(y,x)$ of the graphon $W$ implies that $\mathcal{L}(W) : L^2 \to L^2$ is self-adjoint. One may use the uniform boundedness of $\|W\|_p$ to verify that $\|\mathcal{L}(W)\|_{p\to q} < \infty$ for all $p,q\in[1,\infty]$, thus making the graphon Laplacian operator continuous from any $L^p$ space into another. 

Following \cite{chiba19}, we will use a graphon $W$ to construct finite graphs with $N$ vertices in two different ways. In both cases the unit interval is discretized with $x_j=\frac{j}{N}$ with $j = 1,\dots, N$. A {\em deterministic weighted graph} with adjacency matrix $\gA_d^N = [A_{i,j}^N]_{1 \leq i,j\leq N}$ is obtained by setting 
\be
	A_{ij}^N = \begin{cases}
		W\left(\frac{i}{N},\frac{j}{N}\right) & i \neq j \\
		0 & i = j
		\end{cases}.
\ee
The restriction that $A_{i,i}^N = 0$ means that the associated graph does not contain any loops (edges that originate and terminate at the same vertex). This restriction is merely for the ease of presentation since it can be observed in \eqref{GraphonLap} that the diagonal terms have no effect in the graphon Laplacian. Similarly, these loops have no effect when moving to a finite graph and considering the associated (combinatorial) graph Laplacian
\be\label{detLap}
	\gL_d^N := \gA_d^N - \mathrm{Deg}(\gA_d^N),
\ee  
where $\mathrm{Deg}(\gA_d^N)$ is the diagonal degree matrix whose entries are the sums of the rows of $\gA_d^N$. 

Alternatively, we may obtain a {\em random graph} by considering a symmetric binary adjacency matrix $\gA_r^N = [\xi_{i,j}^N]_{1 \leq i,j\leq N}$ whose elements $\xi_{i,j}^N \in \{0,1\}$ with $i > j$ are {\color{black} independent} random variables with probability distribution 
\be\label{eq:random} 
	\mathbb{P}\left( \xi_{i,j}^N=1\right) = 1 - \mathbb{P}\left( \xi_{i,j}^N=0 \right) = W\left(\frac{i}{N},\frac{j}{N}\right), \qquad \xi_{i,j}^N = \xi_{j,i}^N
\ee
for $i > j$ and $\xi_{i,i}^N = 0$ for all $1 \leq i \leq N$. Setting the diagonal elements of $A_r^N$ to zero again has the effect that we do not consider loops in the resulting graph. As above, this adjacency matrix leads to a graph Laplacian operator associated to a graph with $N$ vertices, denoted
\be\label{randLap}
	\gL_r^N := \gA_r^N - \mathrm{Deg}(\gA_r^N),
\ee 
where $\mathrm{Deg}(\gA_r^N)$ is again the diagonal degree matrix whose entries are the sums of the rows of $\gA_r^N$. It should be noted that we may further consider the points $x_i \in [0,1]$ to be drawn randomly \cite{Borgs,medvedev14} and obtain the deterministic and random graphs as above, but we refrain from doing this for the ease of presentation. The reader should keep in mind that all of the results presented in the following subsections will carry over to the randomly chosen $x_i$, should an application necessitate it.  

Let us consider some motivating examples.

\begin{example} \label{ex:RGmodels} 

\begin{enumerate}
	\item For some $p \in [0,1]$ we can define an {\bf Erd\fH{o}s--R\'eyni} graph using the graphon $W(x,y) = p$ for all $x,y \in [0,1]$. In this case the associated deterministic weighted graph on $N$ vertices has every vertex connected to every other vertex with an edge weight of $p$, while the random graph assigns an edge of weight 1 between any two vertices with probability $p$.	

	\item For parameters $p,q,\alpha \in [0,1]$ the graphon 
	\be
		W(x,y)=\begin{cases}
			 p & \mathrm{if}\ |x-y| \mod 1 \leq \alpha \\ 
			 q & \mathrm{if}\  |x-y| \mod 1 >  \alpha 
			\end{cases}	
	\ee
	is used to generate a {\bf small-world}, or {\bf Watts--Strogatz}, graph. The resulting deterministic graph is a ring such that each vertex is connected to all others, but vertices within a distance of $\alpha$ (with respect to the ring distance) have edge weight $p$, otherwise the edge weight is $q$. In the random graph case the probability of having a connection between two vertices is $p$ if they are within a distance of $\alpha$, otherwise the probability is $q$.
	
	\item For some $p \in (0,1]$ and $\alpha \in (0,1)$ a {\bf bipartite} graph can be formed through the graphon 
	\be
		W(x,y) = \begin{cases}
			p & \mathrm{if}\ \min\{x,y\} \leq \alpha,\ \max\{x,y\} > \alpha \\
			0 & \mathrm{otherwise}
		\end{cases}.
	\ee
	The resulting deterministic graph has an edge with weight $p$ between each $x_i \in [0,\alpha]$ and $x_j \in (\alpha,1]$, while the random graph assigns such edges with probability $p$. No edges are present between two $x_i$ that both belong to either $[0,\alpha]$ or $(\alpha,1]$.  
\end{enumerate}
\end{example} 

{\color{black} In this work we will focus exclusively on ``ring'' networks, which are given by the following definition.

\begin{defn}
	A graphon $W(x,y)$ is said to be a {\bf ring graphon} if there exists a 1-periodic function $R: [0,1] \to [0,1]$ such that $W(x,y) = R(|x - y|)$.   
\end{defn}

Ring graphons have the property that 
\begin{equation}\label{DegreeFn}
	\mathrm{Deg}(x) = \int_0^1 R(|x - y|) \rmd y = \int_0^1 R(|y|) \rmd y, \quad \forall x\in [0,1], 
\end{equation} 
by the periodicity of $R(|x - y|)$ induced by considering the argument modulo 1. Hence, the degree function $\mathrm{Deg}(x)$ is independent of $x \in [0,1]$, which is the graphon analogue of a regular graph. A similar argument shows that for the induced deterministic graph we have $\mathrm{Deg}(\gA_d^N)$ is a constant multiple of the identity, thus making it a degree-uniform graph. The election to study ring networks in this work is motivated by both analysis and application. First, our analysis in the following subsections, particularly Corollary~\ref{cor:CloseGraph} and Theorem~\ref{thm:Random}, heavily rely on the fact that  $\mathrm{Deg}(x)$ is constant. Second, the periodicity of ring graphons means that it can be expanded in a Fourier series, which we will show in Section~\ref{sec:TuringonW} allows one to easily compute the eigenvalues of the associated graphon Laplacian. Third, Erd\fH{o}s--R\'eyni and small-world graphons are some of the most commonly studied graphons, both of which are ring networks, meaning that our work is applicable to both of these graphons. However, the third of our examples, the bipartite graphon, is not a ring graphon and is excluded from the analysis that follows. We note however that the degree function for the bipartite graph is also independent of $x$ when $\alpha = \frac{1}{2}$, and the methods in the following sections can easily be applied here as well. Unfortunately, it is not apparent what the larger class of graphons that the bipartite graphons belong to that lead to the same results in the following subsections and we therefore restrict our attention to ring graphons.    }


\subsection{Deterministic Graphs}\label{subsec:Deterministic}

As an intermediate step to relating the spectral properties of the deterministic weighted graph Laplacian to the graphon Laplacian, we will define a step function graphon that interpolates the discrete domain of the graph to form a step function graphon. To this end we partition $[0,1]$ into $N$ subintervals 
\be
	I_1^{N} = \bigg[0,\frac{1}{N}\bigg), \quad I_2^{N} = \bigg[\frac{1}{N},\frac{2}{N}\bigg), \quad \dots \quad I_N^{N} = \bigg[\frac{N-1}{N},1\bigg).
\ee
{\color{black} Let us now consider a ring graphon $W(x,y) = R(|x - y|)$ and} define a step function $W^N_d:[0,1]^2 \to [0,1]$ by 
\be\label{StepGraphon}
	W^N_d(x,y) =R\bigg(\frac{|i - j|}{N}\bigg)\quad \mathrm{for}\quad (x,y) \in I_i^{N}\times I_j^{N}.
\ee
It is important to note that since we are defining $W^N_d$ from a ring graphon it is itself a degree-uniform graphon for each $N \geq 1$, and so the associated graphon Laplacian, denoted $\mathcal{L}_d^N$, acts by 
\be\label{StepLap}
	\mathcal{L}_d^N f(x) := \int_0^1 W^N_d(x,y) \left(f(y)-f(x)\right)\rmd y, \quad \forall x \in [0,1],
\ee
analogous to \eqref{GraphonLap} above. The relationship between $\mathcal{L}_d^N$ and $\gL_d^N$ is made explicit with the following lemma. 

\begin{lemma} \label{lem:StepEigs} 
	For each $N \geq 2$, $\lambda \in \R$ is an eigenvalue of $\gL_d^N:\R^N \to \R^N$ if and only if $\lambda/N \in \R$ is an eigenvalue of $\mathcal{L}_d^N:L^2 \to L^2$. 
\end{lemma}

\begin{proof}
Throughout this proof we fix $N \geq 2$. Begin by assuming that $(\lambda,\v) \in \R \times \R^N$ is an eigenpair of $\gL_d^N$, i.e. $\gL_d^N\v = \lambda \v$. Then, define the step function 
\be
	f_v(x) = \sum_{i = 1}^{N} v_i \chi_{I_i}(x)
\ee  
where $\v = [v_1,v_2,\dots,v_N]^T$ and $\chi_{I_i}(x)$ is the characteristic function of the interval $I_i^N$. Clearly $f_v$ is square integrable since it takes only finitely many values. Then, for each $1 \leq i \leq N$ and $x \in I_i^N$ we have 
\be\label{StepEig}
	\begin{split}
		\mathcal{L}_d^N f_v(x) &= \int_0^1 W^N_d(x,y) \left(f_v(y)-f_v(x)\right)\rmd y \\ &= {\color{black}\sum_{j = 1}^{N} \frac{1}{N}R\bigg(\frac{|i - j|}{N}\bigg)(v_j - v_i) }\\ &= \frac{1}{N} [\gL_d^N\v]_i \\ &= \frac{\lambda}{N}v_i \\ &= \frac{\lambda}{N}f_v(x).	
	\end{split}
\ee    
Hence, $(\lambda/N,f_v) \in \R\times L^2$ is an eigenpair of $\mathcal{L}_d^N:L^2 \to L^2$, proving the first direction.

Now let us assume that $(\lambda,f) \in \R\times L^2$ is an eigenpair of $\mathcal{L}_d^N: L^2 \to L^2$. Then, for each $1 \leq i \leq N$ and $x \in I_i$ we have 
\be
	\begin{split}
		\lambda f &= \mathcal{L}_d^N f(x) \\
		&= \int_0^1 W^N_d(x,y) \left(f(y)-f(x)\right)\rmd y \\
		&= {\color{black} \sum_{j = 1}^{N} R\bigg(\frac{|i - j|}{N}\bigg)\int_{I_j^N} f(y)\rmd y - \frac{1}{N}\sum_{j = 1}^{N} R\bigg(\frac{|i - j|}{N}\bigg)f(x). }
	\end{split}
\ee
Rearranging the above expression gives
\be
	{\color{black} \bigg[\lambda + \frac{1}{N}\sum_{j = 1}^{N} R\bigg(\frac{|i - j|}{N}\bigg)\bigg]f(x) = \sum_{j = 1}^{N} R\bigg(\frac{|i - j|}{N}\bigg)\int_{I_j^N} f(y)\rmd y, }	
\ee
and since the right-hand-side is independent of $x\in I_i^N$, it follows that $f(x)$ is constant on $I_i^N$ for each $1 \leq i \leq N$. Therefore, there exists $v_i \in \R$, $i = 1,2,\dots, N$, such that  
\be
	f(x) = \sum_{i = 1}^{N} v_i \chi_{I_i^N}(x).
\ee
Defining $\v = [v_1,v_2,\dots,v_N]^T \in \R^N$, it now follows from simply rearranging the equalities in \eqref{StepEig} that $(N\lambda, \v) \in \R\times\R^N$ is an eigenpair of $\gL_d^N:\R^N\to\R^N$. This completes the proof.
\end{proof} 

Lemma~\ref{lem:StepEigs} shows that there is a one-to-one correspondence between the eigenvalues of $\mathcal{L}_d^N$ and $\gL_d^N$. Therefore, our goal in the remainder of this section is to show that the spectrum of $\mathcal{L}_d^N$ converges to the spectrum of $\mathcal{L}(W)$ as $N \to \infty$, allowing one to view the eigenvalues of $\mathcal{L}(W)$ as the limit of eigenvalues of $\gL_d^N$ (after rescaling by $1/N$). This will be achieved by showing that $\mathcal{L}_d^N$ converges to $\mathcal{L}(W)$ in the $L^2 \to L^2$ operator norm, which in turn gives the convergence of the spectra. We present the following lemma that allows one to estimate the distance between graphon Laplacians in the operator norm topology, for which our desired convergence result is a corollary. {\color{black} Note that this result does not require the graphons to be ring graphons, while Corollary~\ref{cor:CloseGraph} below will build upon this lemma in the particular case of almost everywhere continuous ring graphons.} 

\begin{lemma}\label{lem:CloseGraph} 
	Let $W_1$ and $W_2$ be graphons with associated graphon Laplacians $\mathcal{L}(W_1)$ and $\mathcal{L}(W_2)$. Then, 
	\be
		\|\mathcal{L}(W_1) - \mathcal{L}(W_2)\|_{2\to 2} \leq \sqrt{2}\|W_1 - W_2\|_1^{1/2} + \|\mathrm{Deg}(W_1) - \mathrm{Deg}(W_2)\|_\infty. 
	\ee 
\end{lemma}

\begin{proof}
	Since $\mathcal{L}(W_1)$ and $\mathcal{L}(W_2)$ are linear operators, it follows that 
	\be
		\begin{split}
			(\mathcal{L}(W_1) - \mathcal{L}(W_2))f(x) &= \int_0^1 [W_1(x,y) - W_2(x,y)] \left(f(y)-f(x)\right)\rmd y \\
			&= \int_0^1 [W_1(x,y) - W_2(x,y)]f(y)\rmd y - [\mathrm{Deg}(W_1)(x) - \mathrm{Deg}(W_2)(x)]f(x), 
		\end{split}
	\ee
	for all $x \in [0,1]$. Then, using H\"older's inequality we can find that for all $f \in L^\infty$ we have 
	\be 
		\begin{split}
			\bigg\|\int_0^1 [W_1(x,y) - W_2(x,y)]f(y)\rmd y\bigg\|_1 &\leq \|W_1 - W_2\|_1\|f\|_\infty.
		\end{split}
	\ee
	Hence, denoting the linear operator $f \mapsto \int_0^1 [W_1(x,y) - W_2(x,y)]f(y)\rmd y$ by $T$ gives the operator norm bound $\|T\|_{\infty \to 1} \leq \|W_1 - W_2\|_1$. From \cite[Lemma~E.6]{Janson} we get 
	\begin{equation}
		\|T\|_{2 \to 2} \leq \sqrt{2}\|T\|_{\infty \to 1}^{1/2} \leq \sqrt{2}\|W_1 - W_2\|_1^{1/2}.
	\end{equation} 

Then, for all $f \in L^2$ we have that 
	\begin{equation}
		\begin{split}
			\|(\mathcal{L}(W_1) - \mathcal{L}(W_2))f\|_2 &\leq \|Tf\|_2 + \|[\mathrm{Deg}(W_1) - \mathrm{Deg}(W_2)]f\|_2 \\
			&\leq \sqrt{2}\|W_1 - W_2\|_1^{1/2}\|f\|_2 + \bigg(\int_0^1 |\mathrm{Deg}(W_1)(x) - \mathrm{Deg}(W_2)(x)|^2|f(x)|^2\rmd x\bigg)^\frac{1}{2} \\
			&\leq \sqrt{2}\|W_1 - W_2\|_1^{1/2}\|f\|_2 + \|\mathrm{Deg}(W_1) - \mathrm{Deg}(W_2)\|_\infty \bigg(\int_0^1 |f(x)|^2\rmd x\bigg)^\frac{1}{2} \\ 
			&= \bigg(\sqrt{2}\|W_1 - W_2\|_1^{1/2} + \|\mathrm{Deg}(W_1) - \mathrm{Deg}(W_2)\|_\infty\bigg)\|f\|_2.
		\end{split}
	\end{equation}
This therefore gives the desired result.
\end{proof} 

Lemma~\ref{lem:CloseGraph} leads to the following corollary that states that $\mathcal{L}_d^N$, defined in \eqref{StepLap} from any almost everywhere continuous ring graphon, converges in the $L^2\to L^2$ operator norm to $\mathcal{L}(W)$. As one will see in the following proof, in the case of general graphons $W(x,y)$ we require that $\mathrm{Deg}(W^N_d) \to \mathrm{Deg}(W)$ in the $L^\infty$ norm, which may not always be the case. We do point out that bipartite graphons do in fact satisfy this property since the resulting degree matrices are piecewise constant, leading to the belief that the class of graphons that our results apply to is significantly larger than just the ring networks considered in this work.    

\begin{corollary}\label{cor:CloseGraph}
	Let $W$ be a ring graphon {\color{black} that is almost everywhere continuous} and $W^N_d$ defined from $W$ as in \eqref{StepGraphon}. Then the associated graphon Laplacians $\mathcal{L}(W)$ and $\mathcal{L}_d^N$ satisfy $\|\mathcal{L}(W) - \mathcal{L}_d^N\|_{2\to 2} \to 0$ as $N \to \infty$.  
\end{corollary}

\begin{proof}
	By construction of $W^N_d$ we have $W^N_d(x,y) \to W(x,y)$ as $N \to \infty$ for every point of continuity $(x,y)$ of $W(x,y)$. By assumption $W$ is continuous almost-everywhere, and so we have pointwise convergence of $W^N_d$ to $W$ almost-everywhere. Since $0 \leq W \leq 1$, it follows from the dominated convergence theorem that $\|W - W^N_d\|_1 \to 0$ as $N\to \infty$. 
Furthermore, the associated degree functions $\mathrm{Deg}(W)(x)$ and $\mathrm{Deg}(W^N_d)(x)$ are independent of $x$, so 
\begin{equation}
	\begin{split}
		\|\mathrm{Deg}(W) - \mathrm{Deg}(W^N_d)\|_\infty &= |\mathrm{Deg}(W)(x) - \mathrm{Deg}(W^N_d)(x)|  \\ 
		&= \bigg|\int_0^1 [W(x,y) - W^N_d(x,y)] \rmd y\bigg| \\
		&= \int_0^1\bigg|\int_0^1 [W(x,y) - W^N_d(x,y)] \rmd y\bigg|\rmd x \\
		&\leq \int_0^1\int_0^1 |W(x,y) - W^N_d(x,y)| \rmd y\rmd x \\ 
		&= \|W - W^N_d\|_1.
	\end{split}
\end{equation}	
Then, using Lemma~\ref{lem:CloseGraph} we have
	\be
		\|\mathcal{L}(W) - \mathcal{L}_d^N\|_{2\to 2} \leq 2\|W - W^N_d\|_1^{1/2} + \|\mathrm{Deg}(W) - \mathrm{Deg}(W^N_d)\|_\infty \leq 2\|W - W^N_d\|_1^{1/2} + \|W - W^N_d\|_1, 
	\ee
	thus giving that $\|\mathcal{L}(W) - \mathcal{L}_d^N\|_{2\to 2} \to 0$ as $N \to \infty$, completing the proof.
\end{proof} 

Corollary~\ref{cor:CloseGraph} is significant because if $\lambda$ is an isolated eigenvalue of $\mathcal{L}(W)$ with finite multiplicity, it follows that for each associated eigenvector $f \in L^2$ there exists a sequence of eigenpairs $(\lambda_N,f_N) \in \R\times L^2$ of $\mathcal{L}_d^N$ for each $N \geq 2$ converging to $(\lambda,f)$. Furthermore, from Lemma~\ref{StepEig}, the sequence $(\lambda_N,f_N)$ corresponds to a sequence of eigenpairs $(N\lambda_N,\v_N) \in \R\times\R^N$ to $\mathrm{L}_d^N$.


\subsection{Random Graphs}\label{subsec:Random}

In the previous subsection we were able to show that the eigenvalues of the deterministic graph Laplacian limit as $N\to \infty$ to the spectrum of the graphon Laplacian, after weighting by a factor of $\frac{1}{N}$. In this section our goal is to demonstrate the convergence of the eigenvalues and eigenvectors of $\gL_r^N$ to the eigenvalues and eigenvectors of $\gL_d^N$ with overwhelming probability. Our main result stating this fact is summarized in Theorem~\ref{thm:Random} below. The main technical hurdle is to show the convergence of the eigenvectors for large $N$ since we do not have pointwise convergence of the elements in the matrices $\gL_d^N$ and $\gL_r^N$ for arbitrary graphons $W$, and so this precludes applying the Davis--Kahan theorem \cite{DK} which is the traditional tool used to demonstrate eigenvector convergence. We will circumvent this using the results of \cite{AltDK} which replaces the Frobenius norm (Euclidean norm of the elements of a matrix) with a matrix operator norm, defined for a matrix $M$ by
\be\label{OperatorNorm}
	\|M\| := \sup_{\|x\|_2 = 1} \|Mx\|_2,
\ee
{\color{black}where $\|\cdot\|_2$ for a finite dimensional vector is simply the Euclidean norm on the finite dimensional space $\mathbb{R}^N$. We refrain from including the subscript $2 \to 2$ for matrix norms since we will not consider any other domains and ranges and we suppress the dependence on on the spatial dimension $N$ to simply notation.}  

We will begin by bounding $\|\gL_r^N - \gL_d^N\|$ with overwhelming probability. The following lemma is an extension of \cite[Theorem~3.1]{Oliveira} to combinatorial graph Laplacians since the aforementioned result only applies to adjacency matrices and normalized Laplacians. Although we suspect that an extension of these results to combinatorial graph Laplacians may be known in the literature, an appropriate reference could not be found and so we provide a detailed proof here. Finally, we note that this result holds for graph Laplacians generated from any graphon $W$, not just ring graphons.

\begin{lemma}\label{lem:RandConv} 
	Let $N \geq 2$ and consider the graph Laplacians $\gL_d^N$ and $\gL_r^N$ as defined in \eqref{detLap} and \eqref{randLap}, respectively. For all $\gamma \in (0,\frac{1}{2})$, there exists a constant $C = C(\gamma)$, independent of $N$, such that 
	\be
		\mathbb{P}(\|\gL_r^N - \gL_d^N\| \geq N^{\frac{1}{2} + \gamma}) \leq 2N\mathrm{e}^{-CN^{2\gamma}}.
	\ee  
\end{lemma}

{\color{black}The main tool for proving Lemma~\ref{lem:RandConv} will be Corollary~7.1 from \cite{Oliveira}, which we state here as a lemma for the reader. This result can also be seen as a specific application of \cite[Theorem~1.6]{tropp2012user} to square Hermitian matrices.}

\begin{lemma}[\cite{Oliveira}, Corollary~7.1]\label{lem:OlivCor}
	Let $X_1,\dots X_n \in \mathbb{C}^{d\times d}$ be mean-zero independent random Hermitian matrices and suppose that there exists a $M > 0$ with $\|X_i\| \leq M$ almost surely for all $1 \leq i \leq n$. Define:
	\be
		\sigma^2 := \lambda_\mathrm{\max}\bigg(\sum_{i = 1}^n \mathbb{E}[X_i^2]\bigg).
	\ee 
	Then for all $t \geq 0$, 
	\be 
		\mathbb{P}\bigg(\bigg\|\sum_{i = 1}^n X_i\bigg\| \geq t\bigg) \leq 2d\mathrm{e}^{-\frac{t^2}{8\sigma^2 + 4Mt}}.
	\ee
\end{lemma}

\begin{proof}[Proof of Lemma~\ref{lem:RandConv}]
	For a given graphon $W$ and a fixed $N \geq 2$, let us write $p_{i,j} = W(\frac{i}{N},\frac{j}{N})$ for $i \neq j$ and $p_{i,i} = 0$ if $i = j$. Hence, the adjacency matrices can be written 
	\be
		\gA_d^N = [p_{i,j}]_{1 \leq i,j\leq N}
	\ee
	and 
	\be 
		\gA_r^N = [\xi_{i,j}^N]_{1 \leq i,j\leq N}, \quad \mathrm{with}\quad \begin{cases} \mathbb{P}(\xi_{i,j}^N = 1) = 1 - \mathbb{P}(\xi_{i,j}^N = 0) = p_{i,j} & i < j \\
		 \xi_{j,i}^N = \xi_{i,j}^N & i<j \\
			\xi_{i,i}^N = 0 & i = j
		\end{cases}.
	\ee
	Importantly, $\gA_d^N = \mathbb{E}[\gA_r^N]$ and similarly $\gL_d^N = \mathbb{E}[\gL_r^N]$.
	
	Let $\{\ve_i\}_{i = 1}^N$ be the canonical basis for $\R^N$. For each $1 \leq i \leq j \leq N$ define the matrices $B_{i,j} \in \R^{N\times N}$ by
	\be
		B_{i,j} = \begin{cases}
			\ve_i\ve_j^T + \ve_j\ve_i^T, & i \neq j\\
			\ve_i\ve_i^T, & i = j 
		\end{cases}.
	\ee
	Note that $B_{i,j}$ are real-valued, symmetric matrices and that the adjacency matrices can be written 
	\be
		\begin{split}
			\gA_d^N &= \sum_{1 \leq i \leq j \leq N} p_{i,j}B_{i,j}, \\
			\gA_r^N &= \sum_{1 \leq i \leq j \leq N} \xi_{i,j}^NB_{i,j}.
		\end{split}
	\ee
	Consequentially, the associated graph Laplacians can be written 
	\be
		\begin{split}
			\gL_d^N &= \sum_{1 \leq i \leq j \leq N} p_{i,j}(B_{i,j} - B_{i,i} - B_{j,j}), \\
			\gL_r^N &= \sum_{1 \leq i \leq j \leq N} \xi_{i,j}^N(B_{i,j} - B_{i,i} - B_{j,j}).
		\end{split}
	\ee
	Defining 
	\be
		X_{i,j} = (\xi_{i,j}^N - p_{i,j})(B_{i,j} - B_{i,i} - B_{j,j})
	\ee
	gives that 
	\be\label{LapDiff}
		\gL_r^N - \gL_d^N = \sum_{1 \leq i \leq j \leq N} (\xi_{i,j}^N-p_{i,j})(B_{i,j} - B_{i,i} - B_{j,j}) = \sum_{1 \leq i \leq j \leq N} X_{i,j}.
	\ee
	The matrices $X_{i,j} \in \R^{N\times N}$ are independent, symmetric, and have mean zero. Furthermore, since each row of $X_{i,j}$ has at most two nonzero entries that add to zero, it follows that 
{\color{black}	\be\label{XopBnd}
		\|X_{i,j}\| \leq 2|\xi_{i,j}^N - p_{i,j}| \leq 2
	\ee
	for every $1 \leq i \leq j \leq N$.}
	
	Now, 
	\be\label{X2expand}
		\begin{split}
		X_{i,j}^2 &= (\xi_{i,j}^N - p_{i,j})^2(B_{i,j} - B_{i,i} - B_{j,j})^2 \\
		&= (\xi_{i,j}^N - p_{i,j})^2(B_{i,j}^2 + B_{i,i}^2 + B_{j,j}^2 - B_{i,j}B_{i,i} - B_{i,i}B_{i,j} - B_{i,j}B_{j,j} - B_{j,j}B_{i,j} + B_{i,i}B_{j,j} + B_{j,j}B_{i,i}).
		\end{split}
	\ee
	From the definition of the $B_{i,j}$, we have the following properties:
	\begin{enumerate}
		\item $B_{i,i}B_{j,j} = 0$ for all $i \neq j$.
		\item $B_{i,i}^2 = B_{i,i}$ for all $1 \leq i \leq N$.
		\item $B_{i,j}^2 = B_{i,i} + B_{j,j}$ for all $1 \leq i \leq j \leq N$.
		\item $B_{i,j}B_{i,i} + B_{i,i}B_{i,j} = B_{i,j}$ and $B_{i,j}B_{j,j} + B_{j,j}B_{i,j} = B_{i,j}$ for all $1 \leq i \leq j \leq N$.  
	\end{enumerate}
	These properties can in turn be used to simplify \eqref{X2expand}, giving
	\be
		\begin{split}
		X_{i,j}^2 &= (\xi_{i,j}^N - p_{i,j})^2(2B_{i,i} + 2B_{j,j} - 2B_{i,j}) \\
		&= 2(\xi_{i,j}^N - p_{i,j})^2(B_{i,i} + B_{j,j} - B_{i,j}).
		\end{split}
	\ee
	Since the $\xi_{i,j}^N$ are binary random variables, we have the expectation
	\be
		\begin{split}
		\mathbb{E}[X_{i,j}^2]	 &= \underbrace{2p_{i,j}(1 - p_{i,j})^2(B_{i,i} + B_{j,j} - B_{i,j})}_{\xi_{i,j}^N = 1} + \underbrace{2(1 - p_{i,j})p_{i,j}^2(B_{i,i} + B_{j,j} - B_{i,j})}_{\xi_{i,j}^N = 0} \\
		&= 2p_{i,j}(1 - p_{i,j})(B_{i,i} + B_{j,j} - B_{i,j}).
		\end{split}
	\ee 
	Taking the sum over all $1 \leq i \leq j \leq N$ gives
	\be
		\sum_{1 \leq i \leq j \leq N} \mathbb{E}[X_{i,j}^2] = \sum_{1 \leq i \leq j \leq N} 2p_{i,j}(1 - p_{i,j})(B_{i,i} + B_{j,j} - B_{i,j}),
	\ee
	which is the negative of a graph Laplacian with edge weights between vertices with index  $i$ and $j$  given by $2p_{i,j}(1 - p_{i,j})$. Therefore, from the Gershgorin circle theorem we have 
	\be
		\lambda_\mathrm{max}\bigg(\sum_{1 \leq i \leq j \leq N} \mathbb{E}[X_{i,j}^2]\bigg) \leq 2\max_{1 \leq i \leq N} \sum_{j \neq i} 2p_{i,j}(1 - p_{i,j}) \leq 4(N - 1), 
	\ee 
	since $0 \leq p_{i,j}(1 - p_{i,j}) \leq 1$ for all $i\leq j$.
	
	{\color{black} Recall from \eqref{LapDiff} that $\gL_r^N - \gL_d^N = \sum_{1 \leq i \leq j \leq N} X_{i,j}$. We may now apply Lemma~\ref{lem:OlivCor} with $d = N$, $n = N(N+1)/2$, $M = 2$, and $\sigma^2 = 4(N-1)$ to find that for all $t > 0$ we have}
	\be
		\mathbb{P}\bigg(\bigg\|\sum_{1 \leq i \leq j \leq N} X_{i,j}\bigg\| \geq t\bigg) \leq 2N\mathrm{e}^{-\frac{t^2}{32N + 8t - 32}}.
	\ee
	For any $\gamma > 0$, take $t = N^{\frac{1}{2} + \gamma}$. Therefore, we have
	\be
		\mathbb{P}\bigg(\|\gL_r^N - \gL_d^N\| \geq N^{\frac{1}{2} + \gamma}\bigg) \leq 2N\mathrm{e}^{-\frac{N^{1 + 2\gamma}}{32N + 8N^{1/2 + \gamma} - 32}} = 2N\mathrm{e}^{-\frac{N^{2\gamma}}{32 + 8N^{\gamma - 1/2} - 32N^{-1}}}.	
	\ee
{\color{black} Finally, consider any $\gamma\in(0,\frac{1}{2})$.  Notice that $N \geq 1$ (used for simplicity over $N \geq 2$) implies 
\begin{equation}
	0\leq 32 + 8N^{\gamma - 1/2} - 32N^{-1} \leq 32 + 8N^{\gamma - 1/2} \leq 40, 
\end{equation}
from the fact that $N^{\gamma - 1/2} \leq 1$ because $0 < \gamma < 1/2$, and $32 - 32N^{-1} \geq 0$ for all $N\geq 1$. Thus, dividing through by $N^{2\gamma} \geq 1$ gives 
\begin{equation}
	0 \leq \frac{32 + 8N^{\gamma - 1/2} - 32N^{-1}}{N^{2\gamma}} \leq \frac{40}{N^{2\gamma}}.
\end{equation}
Inverting the above fractions and multiplying through by $-1$ preserves the inequality. That is, 
\begin{equation}
	-\frac{N^{2\gamma}}{32 + 8N^{\gamma - 1/2} - 32N^{-1}} \leq -\frac{N^{2\gamma}}{40}.
\end{equation}
Since the exponential function preserves these inequalities, we arrive at the desired result. }
Hence, for all $\gamma\in(0,\frac{1}{2})$ we can bound the exponential bound above by $\mathrm{e}^{-C(\gamma)N^{2\gamma}}$ for some $C(\gamma) > 0$ and {\color{black} any} $N \geq 2$, thus giving the desired bound and completing the proof. 
\end{proof}

Lemma~\ref{lem:RandConv} shows that as the number of vertices $N$ increases, the operator norm of the difference between the deterministic graph Laplacian and the random graph Laplacian is no greater than $N^{\frac{1}{2} + \gamma}$ with  overwhelming probability, for all sufficiently small $\gamma > 0$. {\color{black} As stated at the outset of this subsection, our goal is to understand the behaviour of the eigenvalues and eigenvectors of $\gL_r^N$ as they relate to the eigenvalues of $\gL_d^N$ for large $N$, which we will show can be approximated by isolated eigenvalues of the ring graphon Laplacian $\mathcal{L}(W)$. Hence, it is important to understand the spectrum of a ring graphon Laplacian. To this end, suppose $W$ is a ring graphon in that $W(x,y) = R(|x - y|)$ for some 1-periodic function $R$. Since $W \in L^2$ the Riesz--Fischer theorem guarantees that it can be expressed as the convergent Fourier series
\be\label{GraphonFourier}
	W(x,y)=\sum_{k\in \mathbb{Z}} c_k e^{2\pi\mbi k(x-y) }, 
\ee
for some square-summable $c_k\in\mathbb{R}$. Using the above Fourier series representation for $W$ and the fact that the $e^{2\pi\mbi kx }$ are orthogonal in $L^2$ for different $k \in \mathbb{Z}$, one finds that the eigenvalues of the corresponding graphon Laplacian $\mathcal{L}(W)$ are given explicitly by $\lambda_k = c_k - c_0$ with corresponding eigenfunctions $e^{2\pi\mbi kx}$. Since we have assumed that $W$ is real-valued and symmetric it follows that $c_k = c_{-k}$, meaning that $\lambda_k = \lambda_{-k}$ for all $k \in \mathbb{Z}$. Hence, nonzero eigenvalues come in pairs with (at least) two-dimensional eigenspaces spanned by $e^{2\pi k \mbi x}$ and its complex conjugate. Furthermore, since the $c_k$ are square summable, we have that $c_k \to 0$ as $k \to \pm\infty$, thus giving an accumulation of eigenvalues at $-c_0$ because $\lambda_k \to -c_0$ as $k \to \pm\infty$.
}

{\color{black} Having now characterized the spectrum of a ring graphon Laplacian, we will now present the main result of this section. In the proof we will evoke Weyl's inequality, which we now recall. Let $M_1$ and $M_2$ be $d\times d$ Hermitian matrices with eigenvalues $\mu_1 \geq \dots \geq \mu_d$ and $\nu_1 \geq \dots \geq \nu_d$, respectively. Furthermore, suppose that $M_1 - M_2$ has eigenvalues $\rho_1 \geq \dots \geq \rho_d$. Then, Weyl's inequality states that 
\begin{equation}
	\nu_j + \rho_d \leq \mu_j \leq \nu_j + \rho_1, \quad j = 1,\dots,d.
\end{equation}
We use the fact that for all $j = 1,\dots,d$ we have $|\rho_j| \leq \|M_1 - M_2\|$, the operator norm \eqref{OperatorNorm} of $M_1 - M_2$, thus giving the bound between successive eigenvalues of the matrices $M_1$ and $M_2$:
\begin{equation}\label{WeylThm}
	|\mu_j - \nu_j| \leq \max\{|\rho_1|,|\rho_d|\} \leq \|M_1 - M_2\|. 
\end{equation}
We now state our main result of this section.}

\begin{theorem}\label{thm:Random} 
	{\color{black} Let $W$ be an almost everywhere continuous ring graphon and assume $\mu \in \R$ is an isolated eigenvalue of $\mathcal{L}(W)$ with multiplicity $1 \leq m < \infty$. Then, there exists $\delta_0(\mu) > 0$ such that for all $\delta \in (0,\delta_0)$ and $\gamma \in (0,\frac{1}{2})$ there exists $N_0 = N_0(\delta,\gamma) \geq 1$ and $C = C(\gamma) > 0$ such that the following is true {\color{black} for each $N \geq N_0$}:}
	\begin{enumerate}
		\item There are exactly $m$ eigenvalues of $\gL_d^N$, $\lambda_1 \geq \lambda_2 \geq \dots \geq \lambda_m$, such that
			\be
				\bigg|\frac{\lambda_k}{N} - \mu\bigg|<\delta,
			\ee   
		for all $k = 1,\dots,m$.
		\item There are exactly $m$ eigenvalues of $\gL_r^N$, $\hat \lambda_1 \geq \hat\lambda_2 \geq \dots \geq \hat\lambda_m$, such that 
			\be
				\bigg|\frac{\hat\lambda_k}{N} - \mu\bigg|<\delta,
			\ee 
		for all $k = 1,\dots,m$, with probability at least $1 - 2N\mathrm{e}^{-CN^{2\gamma}}$.
		\item If $V = [\v_1,\v_2,\dots,\v_m]\in\R^{N\times m}$ and $\hat V = [\hat \v_1,\hat \v_2,\dots,\hat \v_m]\in\R^{N\times m}$ have orthonormal columns satisfying $\gL_d^N \v_k = \lambda_k \v_k$ and $\gL_r^N \hat \v_k = \hat\lambda_k \hat \v_k$ for $k = 1,\dots,m$, then there exists an orthonormal matrix $\hat O \in \R^{m\times m}$ such that 
	\be
		\|\hat V\hat O - V\|_F \leq \frac{\sqrt{8m}}{(\delta_0 - \delta)N^{1/2 - \gamma}}	
	\ee
	where $\|\cdot\|_F$ is the Frobenius matrix norm, with probability at least $1 - 2N\mathrm{e}^{-CN^{2\gamma}}$. 
	\end{enumerate}
\end{theorem}

\begin{proof}
	{\color{black} Since $\mu$ is an isolated eigenvalue of $\mathcal{L}(W)$, we can identify a $\delta_0 > 0$ sufficiently small so that the open ball in the complex plane centered at $\mu \in\R$ intersects the spectrum of $\mathcal{L}(W)$ only at $\mu$. Then, from our work in \S\ref{subsec:Deterministic}, it follows that for any $\delta \in (0,\delta_0)$ there exists $N_1 \geq 1$, which depends on $\delta$, so that for all $N \geq N_1$ we have that there are exactly $m$ eigenvalues of $\gL_d^N$, written $\lambda_1 \geq \lambda_2 \geq \dots \geq \lambda_m$, such that 
	\begin{equation}
		\bigg|\frac{\lambda_k}{N} - \mu\bigg|< \frac \delta 2, \qquad k = 1,\dots, m,
	\end{equation}
	for all $N \geq N_1$. This holds since $\|\mathcal{L}(W) - \mathcal{L}_d^N\|_{2 \to 2} \to 0$ as $N \to \infty$ which gives the convergence of isolated eigenvalues (with multiplicity) between the two operators, and from Lemma~\ref{lem:StepEigs} we have that the eigenvalues $\lambda$ of $\gL_d^N$ are exactly eigenvalues $\lambda/N$ for $\mathcal{L}_d^N$. 

	We now turn to the spectrum of $\gL_r^N$. From Lemma~\ref{lem:RandConv}, for {\color{black} each fixed} $N \geq 2$ and $\gamma \in (0,\frac 1 2)$, there exists a $C = C(\gamma) > 0$ such that $\|\gL_r^N - \gL_d^N\| \leq N^{\frac{1}{2} + \gamma}$ with probability at least $1 - 2N\mathrm{e}^{-CN^{2\gamma}}$. Since $\gL_r^N$ and $\gL_d^N$ are both symmetric matrices for any $N$, it follows {\color{black} from Weyl's inequality via \eqref{WeylThm} that} with probability at least $1 - 2N\mathrm{e}^{-CN^{2\gamma}}$ there exists $m$ eigenvalues of $\gL_r^N$, denoted $\hat \lambda_1 \geq \hat\lambda_2 \geq \dots \geq \hat\lambda_m$, such that 
	\be
		|\lambda_k - \hat\lambda_k| \leq \|\gL_r^N - \gL_d^N\| \leq N^{\frac{1}{2} + \gamma}, 
	\ee
	for each $k = 1,\dots,m$. Now, let $N_2 \geq 1$ be taken sufficiently large so that $N_2^{\gamma - \frac 1 2} \leq \frac \delta 2$, which exists since $\gamma \in (0,\frac 1 2)$. Therefore, {\color{black} for each fixed} $N \geq N_0 := \max\{N_1,N_2\}$ we have 
	\be
		\begin{split}
			\bigg|\frac{\hat\lambda_k}{N} - \mu\bigg| &\leq \bigg|\frac{\lambda_k}{N} - \mu\bigg| + \bigg|\frac{\hat\lambda_k}{N} - \frac{\lambda_k}{N}\bigg|  \\
			&< \frac{\delta}{2} + \frac{N^{\frac 1 2 + \gamma}}{N} \\
			& \leq \frac{\delta}{2} + \frac{\delta}{2} \\
			&= \delta,	
		\end{split}
	\ee
	for each $k = 1,\dots,m$ with probability at least $1 - 2N\mathrm{e}^{-CN^{2\gamma}}$. The fact that exactly $m$ eigenvalues of $\gL_r^N$ satisfy the above comes from the fact that all other eigenvalues of $\gL_d^N$ for $N$ large are bounded away from $\mu$ and Weyl's inequality again gives that the remaining eigenvalues of $\gL_r^N$ are close to these eigenvalues after normalizing by $N$. This proves statement (2) of the theorem.
	
	We now prove statement (3) of the theorem. By definition of $\delta_0$, it follows that for each $N \geq N_0$, if $\tilde{\lambda} \notin \{\lambda_1,\lambda_2,\dots,\lambda_m\}$ is an eigenvalue of $\gL_d^N$, then 
	\be
		\bigg|\frac{\tilde\lambda}{N} - \mu\bigg| > \delta_0. 	
	\ee 
	Thus, using the reverse triangle inequality, we get 
	\be
		\bigg|\frac{\tilde\lambda}{N} - \frac{\lambda_j}{N}\bigg| \geq \bigg|\frac{\tilde\lambda}{N} - \mu\bigg| - \bigg|\frac{\lambda_j}{N} - \mu\bigg| > \delta_0 - \delta 	
	\ee
	for all $j = 1,\dots,m$. Multiplying through by $N$ implies that 
	\be
		|\tilde\lambda - \lambda_j| > (\delta_0 - \delta)N	
	\ee
	for all $j = 1,\dots, m$ and $\tilde{\lambda} \notin \{\lambda_1,\lambda_2,\dots,\lambda_m\}$ an eigenvalue of $\gL_d^N$. That is, the gap between the $m$ eigenvalues $\lambda_j$ of $\gL_d^N$ and its other eigenvalues grows at a rate of $\mathcal{O}(N)$ as $N \to \infty$.
	
	Now, the variant of the Davis--Kahan theorem presented in \cite[Theorem~2]{AltDK} gives that there exists an orthogonal matrix $\hat O \in \R^{m\times m}$ such that 
	\be
		\|\hat V\hat O - V\|_F \leq \frac{\sqrt{8m}\|\gL_r^N - \gL_d^N\|}{\min\{|\tilde\lambda - \lambda_j|\}},	
	\ee
	where $\hat V$ and $V$ are as in the statement of the lemma and the minimum in the denominator is taken over all $j = 1,\dots,m$ and $\tilde{\lambda} \notin \{\lambda_1,\lambda_2,\dots,\lambda_m\}$ an eigenvalue of $\gL_d^N$. Using Lemma~\ref{lem:RandConv} again, we have that for each $\gamma \in (0,\frac{1}{2})$, $\|\gL_r^N - \gL_d^N\| \leq N^{1/2 + \gamma}$ with probability at least $1 - 2N\mathrm{e}^{-CN^{2\gamma}}$. Hence, the results of statement (3) in the theorem now follows from the inequality
	\be
		\|\hat V\hat O - V\|_F \leq \frac{\sqrt{8m}N^{1/2 + \gamma}}{(\delta_0 - \delta)N} \leq \frac{\sqrt{8m}}{(\delta_0 - \delta)N^{1/2 - \gamma}},
	\ee  
	which holds with probability at least $1 - 2N\mathrm{e}^{-CN^{2\gamma}}$. This completes the proof of the theorem. }
\end{proof}


\section{Turing Bifurcations for Ring Graphons}\label{sec:TuringonW}

We now employ the machinery introduced in the previous section to understand Turing bifurcations for the class of ring graphon Laplacians wherein the weight of an edge in the network is a function of the distance between the two nodes. In this section, we consider the non-local graphon version of the Swift-Hohenberg equation, given by
\be \label{eq:SHgraphon}
	\frac{du}{dt} = -(\L-\kappa)^2 u +\epsilon u +ru^2-bu^3  
\ee
for some constants $r\in\mathbb{R}$, $b>0$ and  $\kappa<0$.  Recall that the parameter $r$ controls the magnitude of the odd-symmetry breaking quadratic term while $\kappa$ and $\e$ control which mode(s) $e^{2\pi k \mbi x}$ are unstable. The linearization of the right-hand-side of \eqref{eq:SHgraphon} about the zero state, $u = 0$, with $\e=0$ is $\mathrm{S}_W=-(\L-\kappa)^2$ which is a operator on  $L^2$  with eigenvalues 
\be \label{graphLapEig}
	\ell_k=-\lambda_k^2+2\kappa\lambda_k-\kappa^2.
\ee
{\color{black} If we fix $\kappa = \lambda_k$} for some $k$ then the zero state in \eqref{eq:SHgraphon} will become unstable for any $\e>0$ and we expect that a non-zero stationary solution will bifurcate.  Our goal in this section is to employ center manifold theory to characterize the bifurcation that occurs as $\e$ passes through zero.  

{\color{black} The basic idea of the following analysis is to show that when the homogeneous state loses stability then generic initial conditions starting near the zero state will grow and converge to some other state.  In fact, the loss of stability is associated with the creation of nearby stationary states and by performing a center manifold reduction these bifurcating states can be identified.  Several scenarios are possible and classifying them provides important information regarding the state of the system near the bifurcation.  In a {\em pitchfork} bifurcation, locally, two symmetrically-related steady-states are created.  These states can exist when $\e>0$ ({\em super-critical}) so that small perturbations of homogeneous state will typically converge to a small-amplitude steady state that resembles the destabilizing mode.  On the other hand, if the bifurcating states exist for $\e<0$  ({\em sub-critical}) then a very different situation occurs: small perturbations of the homogeneous state will typically converge to a large amplitude solution far from equilibria that may or may not resemble the destablizing mode.  Pitchfork bifurcations rely on the lack of quadratic interactions between the bifurcating solutions.  When such interactions do occur then the bifurcation is generally {\em transcritical} where a single branch of equilibrium exist for $\e>0$.  In both the pitchfork and transcritical case the bifurcating solution resembles the destabilizing mode, but whether this state is selected or the system transitions to some other large-amplitude state near the bifurcation depends on the particular initial conditions.    }

Since the eigenfunctions in this case are Fourier basis functions, the analysis will bear some resemblance to the case of Turing bifurcations that occur for PDEs with local spatial interactions coming from the spatial derivatives.  In Section \ref{sec:graphonsinglemode} that follows, we consider the case when a single pair of complex conjugate modes bifurcate and classify the bifurcation as a sub- or super-critical pitchfork, depending on the choice of parameters. In contrast, we show in Section \ref{sec:graphonresonance} that if the bifurcating eigenvalue is not simple then the resulting bifurcation may resemble a transcritical bifurcation, depending on whether there exists a resonance between the two bifurcating modes.

Prior to moving on to our analysis, we point out that many other bifurcation scenarios are possible beyond those that we cover in the following subsections. In such cases, bifurcating modes may or may not be in resonance and, depending on the particulars of their interactions, we could obtain center manifold reductions akin to those obtained in the following subsections.  We do not pursue these calculations here as these further scenarios are not generic. {\color{black} Our intention with the non-simple bifurcation scenario covered in Section \ref{sec:graphonresonance} is simply to illustrate how the center manifold reductions would be carried out for these more complex scenarios and to demonstrate that the bifurcations are fundamentally different than the simple bifurcation case covered in Section \ref{sec:graphonsinglemode}.}

{\color{black} Finally, a more exotic scenario can occur when the bifurcation involves the accumulation point at $-c_0$ from the Fourier expansion \eqref{GraphonFourier}, coming from the fact that $\lambda_k \to -c_0$ as $k \to \pm\infty$.} If $\kappa = -c_0$ the bifurcation at $\e = 0$ has infinite co-dimension and the methods of center manifold theory do not apply. Therefore, we are unable to classify the bifurcation rigorously in this scenario and leave it for a potential follow-up investigation. However, we will return to this case in Section~\ref{sec:numerics} where we perform a numerical investigation from an eigenvalue of a random graph near the cluster point of the graphon Laplacian (after rescaling by the number of vertices).


\subsection{Case I: Bifurcation through a single mode} \label{sec:graphonsinglemode} 

In this section, we prove the following result that describes the bifurcation for a single mode. We recall throughout that $\ell_k$ are given in \eqref{graphLapEig} and correspond to the eigenvalues of the graphon Laplacian.   

\begin{theorem}\label{thm:graphonsinglemode} 
Fix $\kappa<0$ and suppose that there exists a $k^*\in\mathbb{N}$ such that $\ell_{k^*}=\ell_{-k^*}=0$ while $\ell_k<0$ for all $k\neq \pm k^*$.  Then for all $|\e|$ sufficiently small \eqref{eq:SHgraphon} has an attracting two-dimensional center manifold in a neighborhood of $u = 0$ and the flow restricted to this center manifold is topologically equivalent to the system of ordinary differential equations
\be \label{eq:case1CM} 
	\begin{split}
		\frac{dw_{k^*}(t)}{dt}&= \e w_{k^*}  -\left(3b+\frac{4r^2}{\ell_0}+\frac{2r^2}{\ell_{2k^*}}   \right) w_{k^*}^2 w_{-k^*}   +\O(4) \\
		\frac{dw_{-k^*}(t)}{dt}&= \e w_{-k^*} -\left(3b+\frac{4r^2}{\ell_0}+\frac{2r^2}{\ell_{2k^*}} \right) w_{k^*}w_{-k^*}^2+\O(4).
	\end{split}
\ee
If  $3b+4\frac{r^2}{\ell_0}+2\frac{r^2}{\ell_{2k^*}}>0$, then the bifurcation at $\e = 0$ is super-critical and there exists an  $\e_0>0$ such that for all $0<\e<\e_0$ and $\phi \in \mathbb{R}$ there exists  a {\color{black} stable} stationary solution $u_{\e,\phi}$ of \eqref{eq:SHgraphon} with expansion 
\be \label{eq:uphi} 
	u_{\e,\phi}(x)=2\sqrt{\frac{\e}{3b+4\frac{r^2}{\ell_0}+2\frac{r^2}{\ell_{2k^*}}}}\cos\left(2\pi k^*(x-\phi)\right)+\O(\e). 
\ee
On the other hand, if $3b+4\frac{r^2}{\ell_0}+2\frac{r^2}{\ell_{2k^*}}<0$ then the bifurcation is sub-critical and an {\color{black} unstable} stationary solution of \eqref{eq:SHgraphon} of the form \eqref{eq:uphi} exists for all $-\e_0<\e<0$.  
\end{theorem}

\begin{proof}
The derivation of \eqref{eq:case1CM} is an application of the Center Manifold Theorem in infinite dimensions; see for example \cite{haragus}.
We begin with the eigenfunction expansion of a solution $u(t,x)$ of \eqref{eq:SHgraphon} in the form
\be
	u(t,x)=\sum_{k\in\mathbb{Z}} w_k(t,\e) e^{2\pi\mbi k x}. 
\ee
{\color{black} For an arbitrary element $\phi \in L^2$}, define the projection operator  $P_*:L^2\to L^2$ by
\be \label{Projection}
	P_* \phi= \langle \phi, e^{2\pi\mbi k_* x} \rangle e^{2\pi\mbi k_* x} +\langle \phi, e^{-2\pi\mbi k_* x} \rangle e^{-2\pi\mbi k_* x},
\ee
where $\langle \cdot,\cdot\rangle$ is the standard Hermitian inner product on the Hilbert space $L^2$. Since the range of $P_*$ is finite-dimensional, it is closed, thus leading to the decomposition of $u$ into $u_*=P_* u$ and $u_\perp=(\mathrm{I}-P_*)u$. Applying $P_*$ and $(\mathrm{I} - P_*)$ to \eqref{eq:SHgraphon} results in the equations 
\be\label{graphonSplit}
	\begin{split}
		\frac{du_*}{dt}&= \epsilon u_* + P_* ( ru^2-bu^3) \\
		\frac{du_\perp}{dt}&= S_\perp u_\perp +\epsilon u_\perp+ \left(\mathrm{I}-P_*\right) ( ru^2-bu^3) \\
		\frac{d\e}{dt} &=0 
	\end{split}
\ee
where $S_\perp$ is the restriction of $-(\L - \kappa)^2$ onto the range of $\mathrm{I}-P_*$. {\color{black} Note that we have trivially included the parameter $\e$ as a variable, as is standard with an application of the center manifold theorem}. By assumption, the spectrum of $S_\perp$ is bounded away from the imaginary axis, and therefore center manifold theory implies the existence of a local, flow-invariant center manifold parametrized as $u_\perp=\Psi(u_*,\epsilon)$ \cite{haragus}. The reduced flow on the {\color{black} ($\e$-dependent)} two dimensional center manifold is then given by 
\be
	\frac{du_*}{dt}= \epsilon u_* +P_* \mathcal{N}(u_*+\Psi(u_*,\epsilon)), \label{eq:CMformal} 
\ee
after putting $u_\perp=\Psi(u_*,\epsilon)$ into the first equation in \eqref{graphonSplit} and letting the function $\mathcal{N}$ simply denote all nonlinear terms on the right-hand-side.  To compute the reduced flow on the center manifold we therefore must obtain expansions for $\Psi(u_*,\e)$ and then expand the projection onto the center eigenspace in (\ref{eq:CMformal}).  

We now compute leading order nonlinear terms in $\mathcal{N}$ above.  Using the fact that $u_* = P_*u$ and the definition of $P_*$ in \eqref{Projection}, the equations for the critical bifurcating modes $w_{k^*}$ and $w_{-k^*}$ are 
\be
	\begin{split}
		\frac{dw_{k^*}(t)}{dt}&= \e w_{k^*}+r \sum_{k_1+k_2=k^*} w_{k_1}w_{k_2} -b \sum_{k_1+k_2+k_3=k^*} w_{k_1}w_{k_2}w_{k_3} \\
		\frac{dw_{-k^*}(t)}{dt}&= \e w_{-k^*}+r \sum_{k_1+k_2=-k^*} w_{k_1}w_{k_2} -b \sum_{k_1+k_2+k_3=-k^*} w_{k_1}w_{k_2}w_{k_3} .
	\end{split} \label{eq:crit}
\ee
{\color{black} We now compute quadratic expansions for the graph of the center manifold.  Using $w_{k^*}$ and $w_{-k^*}$ as coordinates on the center manifold  (and  recalling that we treat $\e$ as a dependent variable) we express the $k$-th component of the graph $\Psi(u_*,\e)$  as  $H_k(w_{k^*},w_{-k^*},\e)$ with the expansion (using multi-index notation)
\be 
	H_k(w_{k^*},w_{-k^*},\e)=\sum_{|\alpha|\geq 2} \zeta_{k,\alpha} w_{k^*}^{\alpha_1} w_{-k^*}^{\alpha_2}\e^{\alpha_3},  \label{eq:wkgraph}
\ee
where $\alpha=(\alpha_1,\alpha_2,\alpha_3)$.  In what follows we will require computation of some of the quadratic  coefficients $\zeta_{k,\alpha}$.  The coefficients can be computed as follows:  substitute (\ref{eq:wkgraph}) for $k\neq \pm k^*$ into the differential equation for $w_k$,
\be \frac{dw_{k}(t)}{dt}=\ell_0 w_k+ \e w_{k}+r \sum_{k_1+k_2=k} w_{k_1}w_{k_2} -b \sum_{k_1+k_2+k_3=k} w_{k_1}w_{k_2}w_{k_3}, \label{eq:wk} \ee
expand each side as a Taylor series in $w_{k^*}$, $w_{-k^*}$ and $\e$ and then match coefficients to solve for each $\zeta_{k,\alpha}$.  

We begin with computing the quadratic terms.  For any $k\neq \pm k^*$, substitution of (\ref{eq:wkgraph}) into the left side of (\ref{eq:wk}) yields the terms 
\be \sum_{|\alpha|\geq 2}\zeta_{k,\alpha}\left(\alpha_1 w_{k^*}^{\alpha_1-1} w_{-k^*}^{\alpha_2}\e^{\alpha_3}\frac{dw_{k^*}}{dt}+ \alpha_2 w_{k^*}^{\alpha_1} w_{-k^*}^{\alpha_2-1}\e^{\alpha_3}\frac{dw_{-k^*}}{dt}\right), \label{eq:quadleft} \ee
where we have used that $\frac{d\e}{dt}=0$.  Recall from (\ref{eq:crit}) that all terms in the differential equation for the modes $w_{k^*}$ and $w_{-k^*}$ are of quadratic or higher order.  Therefore the terms in (\ref{eq:quadleft}) are all of cubic or higher order and not relevant to the computation of the quadratic coefficients defining the center manifold.  Upon substituting (\ref{eq:wkgraph}) into the right hand side of (\ref{eq:wk}) we get have 
\begin{eqnarray} (\ell_k+\e) H_k(w_{k^*},w_{-k^*},\e)&+&r \sum_{k_1+k_2=k} H_{k_1}(w_{k^*},w_{-k^*},\e)H_{k_2}(w_{k^*},w_{-k^*},\e)  \nonumber \\
&-&b \sum_{k_1+k_2+k_3=k} H_{k_1}(w_{k^*},w_{-k^*},\e)H_{k_2}(w_{k^*},w_{-k^*},\e)H_{k_3}(w_{k^*},w_{-k^*},\e). \label{eq:quadright} \end{eqnarray}
Again focusing on quadratic terms, we observe that if $k\neq 0, \pm 2k^*$ then the only quadratic terms arising in (\ref{eq:quadright}) are $\ell_k \zeta_{k,\alpha}w_{k^*}^{\alpha_1}w_{-k^*}^{\alpha_2}\e^{\alpha_3}$ and therefore matching quadratic terms on the left and right requires $\zeta_{k,\alpha}=0$ for $k\neq 0, \pm 2k^*$ and any $|\alpha|=2$.  

Now turning to the modes $k=0$ and $k=\pm2k^*$, (\ref{eq:quadright}) we see that the graph of the center manifold for these components have the following expansions
\be\label{eq:w0}
	\begin{split}
 		w_0&= \zeta_{0,(2,0,0)}w_{k^*}^2+ \zeta_{0,(1,1,0)} w_{k^*} w_{-k^*} +\zeta_{0,(0,2,0)} w_{-k^*}^2 +\O(3) \\
 		w_{\pm 2k^*} &= \zeta_{\pm 2k^*,(2,0,0)}w_{k^*}^2+ \zeta_{\pm 2k^*,(1,1,0)} w_{k^*} w_{-k^*} +\zeta_{\pm 2k^*,(0,2,0)} w_{-k^*}^2 +\O(3). 
	\end{split}
\ee
To compute the remaining coefficients we again use (\ref{eq:quadright}).  For $w_0$, using $(k_1,k_2)=(k^*,-k^*)$ and $(k_1,k_2)=(-k^*,k^*)$  we find that the quadratic terms in (\ref{eq:quadright}) are
\be \ell_0\left( \zeta_{0,(2,0,0)}w_{k^*}^2+ \zeta_{0,(1,1,0)} w_{k^*} w_{-k^*} +\zeta_{0,(0,2,0)} w_{-k^*}^2\right) +2r w_{k^*}w_{-k^*}+\O(3). \ee
We therefore obtain  $\zeta_{0,(2,0,0)}=\zeta_{0,(0,2,0)}=0$ while 
\be \zeta_{0,(1,1,0)}=-\frac{2r}{\ell_0}. \ee
A similar analysis obtains  $\zeta_{\pm 2k^*,(1,1,0)}=\zeta_{- 2k^*,(2,0,0)}=\zeta_{2k^*,(0,2,0)}=0$ and 
\be
	\begin{split}
	 	\zeta_{2k^*,(2,0,0)} &= -\frac{r}{\ell_{2k^*}} \\
 		\zeta_{-2k^*,(0,2,0)}&=- \frac{r}{\ell_{2k^*}} .
	\end{split}
\ee
Equation (\ref{eq:quadright}), together with the quadratic expansions of the graph of the center manifold, could be used to compute cubic and higher order terms in (\ref{eq:wkgraph}), but these will not be required for our analysis.  We remark that the coefficients $\zeta_{k,\alpha}$ may equivalently  be obtained through normal form computations wherein terms $w_{k^*}^{\alpha_1}w_{-k^*}^{\alpha_2}\e^{\alpha_3}$ are removed from the right hand side of the differential equation through a sequence of near-identity coordinate changes that serve to map -- order by order --  the center manifold to the center subspace; see for example \cite{chow94}.   }

We now proceed to compute the flow on the center manifold.  This is obtained by substituting the graph $w_k=H_k(w_{k^*},w_{-k^*},\e)$ into the differential equation for the modes $w_{\pm k^*}$; see (\ref{eq:crit}). Recall that all $H_k$ are quadratic in the center manifold coordinates $w_{\pm k^*}$ and $\e$.  Inspecting the terms in the summations, we see that the quadratic terms $w_{k^*}^2$, $w_{k^*}w_{-k^*}$ and $w_{-k^*}^2$ are all absent and therefore owing to the quadratic dependence of the center manifold expansions we have that the nonlinearity for the reduced equation on the center manifold involves exclusively terms that are cubic order or higher.  In the equation for $w_{k^*}$, one of these cubic terms comes from the interaction $(k_1,k_2,k_3)= (k^*,k^*,-k^*)$ and after including all permutations of indices we obtain the $-3b$ term in the center manifold reduction \eqref{eq:case1CM}.  Similar computations involving the quadratic interaction terms with $(k_1,k_2) = (0,k^*)$ and $(k_1,k_2) = (2k^*,-k^*)$ yield the remaining cubic terms in the center manifold expansion, given in \eqref{eq:case1CM}. 

Since $u(x,t)$ is real, we have $w_{k^*}=\overline{w_{-k^*}}$ and so letting $w_{k^*}(t)=z(t) e^{-2\pi k^* \mbi\phi}$ leads to a {\color{black} scalar differential equation} for $z(t)$,
\be\label{eq:scalarz}  
	\frac{dz}{dt}=\e z -\left(3b+4\frac{r^2}{\ell_0}+2\frac{r^2}{\ell_{2k^*}} \right)z^3 +\O(z \e^2,z^4). 
\ee
Let 
\be
	\Gamma := 3b+4\frac{r^2}{\ell_0}+2\frac{r^2}{\ell_{2k^*}}. 
\ee
If $\Gamma>0$, the implicit function theorem gives the existence of an equilibrium solution of \eqref{eq:scalarz}, valid for sufficiently small $\e>0$, with expansion
\be 
	z_\pm(\e)=\pm \sqrt{\frac{\e}{\Gamma}} +\O(\e).
\ee
In this case that the resulting bifurcation is a super-critical pitchfork. Since $z$ is an amplitude we consider only the root $z_+(\e)$ and after reverting coordinates we obtain the result of the theorem.  If $\Gamma<0$ then the same formula holds but is valid only for $\e<0$ and the bifurcation is sub-critical. This completes the proof. 
\end{proof} 


\subsection{Case II: Bifurcation through repeated modes in resonance } \label{sec:graphonresonance}

Let us now consider the case when two modes simultaneously destabilize.  This occurs when there exists a $k_1$ and a $k_2$ for which $\ell_{k_1}=\ell_{-k_1}=\ell_{k_2}=\ell_{-k_2}=0$, as given in \eqref{graphLapEig}. In this subsection we will primarily focus on the case that the two modes are in $2:1$ resonance, that is $k_2=2k_1$. The following result shows that this case leads to quadratic terms in the reduced equations on the center manifold. 

\begin{theorem}  \label{thm:graphonresonance} 
Suppose $r \neq 0$ and suppose that for $\kappa < 0$ fixed there exists a pair $0<k_1<k_2$ such that $\ell_{k_1}=\ell_{-k_1}=\ell_{k_2}=\ell_{-k_2}=0$ and $\ell_k< 0$ for all other $k$.   

a) If $k_2=2k_1$, then, for $|\e|$ sufficiently small there exists an attracting four dimensional center manifold in a neighborhood of $u = 0$ whose dynamics are topologically conjugate to the ordinary differential equations 
\be\label{eq:case2CM} 
	\begin{split}
		\frac{dw_{k_1}(t)}{dt}&= \e w_{k_1}+2rw_{k_2}w_{-k_1}  +\O(3)  \\
		\frac{dw_{-k_1}(t)}{dt}&= \e w_{-k_1}+2rw_{-k_2}w_{k_1}  +\O(3)  \\
		\frac{dw_{k_2}(t)}{dt}&= \e w_{k_2}+rw_{k_1}^2 +\O(3)  \\
		\frac{dw_{-k_2}(t)}{dt}&= \e w_{-k_2} +r w_{-k_1}^2 +\O(3).
	\end{split}
\ee
The resulting bifurcation leads to a pair of one-parameter bifurcating solutions with expansions
\be u_{\e,\phi,\pm}(x)=\pm \frac{\sqrt{2}\e}{r} \cos\left(2\pi k_1 (x-\phi)\right) -\frac{\e}{r} \cos\left(4\pi k_1 (x-\phi)\right)+\O(\e^2), \label{eq:resonantexpgraphon} \ee 
for any $\phi\in\mathbb{R}$.  

b) If $0<k_1<k_2$ and $k_2>3k_1$ then for $\e$ sufficiently small there exist an attracting four dimensional center manifold in a neighborhood of the origin.  The center manifold dynamics decouple to cubic order and each pair $w_{\pm k_1}$ and $w_{\pm k_2}$ obey equations analogous to \eqref{eq:case1CM}.
\end{theorem}

\begin{proof}
a) Let $k_1=k_*$, so that $k_2 = 2k_*$.  When $\e=0$, by assumption the spectrum of the origin has four zero eigenvalues while the remainder are strictly negative. Once again, there exits a decomposition of $u=u_*+u_\perp$ with $u_*=P_* u$ and $u_\perp=(\mathrm{I}-P_*)u$ where $P_*$ is the spectral projection onto the center eigenspace, spanned by $e^{2\pi\mbi k_* x}$, $e^{4\pi\mbi k_* x}$, and their complex conjugates. The center manifold theorem again applies and we obtain the existence of a four dimensional attracting center manifold parameterized as $u_\perp=\Psi(u_*,\e)$ \cite{haragus}.  The key point is that $\Psi$ is quadratic in $u_*$ and $\varepsilon$ to leading order.   Therefore, the expansion of the graph of the center manifold will only come into play when computing the cubic terms and we can focus on the quadratic terms which only arise through interactions of the four bifurcating modes.   Since $2k_1=k_2$ we obtain the quadratic term $r w_{\pm k_1}^2$ in the equation for $w_{\pm k_2}$ while $k_2-k_1=k_1$ implies that the term $2rw_{\pm k_2}w_{\mp k_1}$ appears in the equation for $w_{\pm k_1}$.  We therefore obtain \eqref{eq:case2CM}. 

We now analyze the steady-states of \eqref{eq:case2CM}.  Since we are interested in real solutions it holds that $w_{k_1}=\overline{w_{-k_1}}$ and $w_{k_2}=\overline{w_{-k_2}}$.  Equilibrium solutions then satisfy the system of complex equations,
\be
	\begin{split}
		0&=\e w_{k_1}+2rw_{k_2}\overline{w_{k_1}}+\O(3)  \\  
		0&=  \e w_{k_2}+rw_{k_1}^2 +\O(3).
	\end{split}
\ee
Up to quadratic order, solutions to the second equation take the form $w_2=-\frac{\e}{r}w_{k_1}^2$.  Substituting into the first equation and again focusing exclusively on the quadratic terms solutions we have
\be 
	\e w_{k_1}\left(\e^2-2r^2 |w_{k_1}|^2\right)=0. 
\ee
Hence,
\be w_{k_1}=\eta e^{-2\pi\mbi \Omega}, \quad \eta= \frac{1}{\sqrt{2}}\left| \frac{\e}{r}\right| \ee 
and
\be w_{k_2}=\frac{-r}{\e} \eta^2 e^{-4\pi\mbi \Omega}=-\frac{\e}{2r}e^{-4\pi\mbi \Omega}, \ee
for some $\Omega\in\mathbb{R}$.  This leads to a pair of bifurcating solutions that have the expansions given in \eqref{eq:resonantexpgraphon}.   

b) Since $k_2>3k_1$ we find that the bifurcating modes do not interact under quadratic or cubic coupling.  Therefore, for $w_{\pm k_1}$ the only terms that are relevant in deducing the quadratic and cubic terms in the reduced equation on the center manifold are those involving $w_{\pm k_1}$, $w_{\pm 2k_1}$ and $w_0$, giving that the analysis follows as in the proof of Theorem~\ref{thm:graphonsinglemode}.  Likewise, the for $w_{\pm k_2}$ the only terms that are relevant in deducing the quadratic and cubic terms in the reduced equation on the center manifold are those involving $w_{\pm k_2}$, $w_{\pm 2k_2}$ and $w_0$ and, again, the reduced equations are (to cubic order) obtained in a manner analogous to those in Theorem~\ref{thm:graphonsinglemode}. We omit the details. 
\end{proof}  

{\color{black}
\begin{rmk} Theorem~\ref{thm:graphonsinglemode} and Theorem~\ref{thm:graphonresonance} are not exhaustive with a multitude of other scenarios being possible. For example, when $k_2=3k_1$ another cubic term arises is the reduced equations on the center manifold. In principle, the analysis presented here could be extended to many cases that are left out of this analysis should an application necessitate it. The main takeaway from Theorem~\ref{thm:graphonresonance} is that resonant interaction of modes can lead to transcritical bifurcations, which we also expect for many of the more exotic scenarios that are not captured by our results in this section. However, one scenario where the center manifold theorem is not sufficient to describe the resulting bifurcation is in understanding bifurcations involving the accumulations point $-c_0$ of the graphon Laplacian. This presents a significant analytical challenge and it is not clear how to obtain similar results to this in this section for such a degenerate situation.
\end{rmk}

\begin{rmk}
If we contrast with the case of the local Swift-Hohenberg PDE in one space dimension we note that it is not typically the case that the inclusion of quadratic terms will lead to the occurrence of a transcritical bifurcation. However, this is not so rare in higher dimensional problems. For example, quadratic terms can lead to a resonance between roll and hexagaonal patterns for the planar Swift--Hohenberg PDE; see for example \cite{cross}.
\end{rmk}}


\section{Turing Bifurcations on Random Graphs} \label{sec:Turinggraph}

In this section we study Turing bifurcations for the Swift--Hohenberg equation \eqref{eq:main} defined on random graphs.  We will consider random graphs of ring networks derived from a ring graphon {\color{black}$W(x,y)=R(|x-y|)$} so that the probability of two nodes having an edge connecting them is a function of the distance between the two nodes.  Based upon the results of Section~\ref{sec:graphon} we expect that when the number of nodes $N$ is sufficiently large then (isolated) eigenvalues and their eigenvectors of the random graph are well approximated by eigenvalues (after scaling) and eigenvectors of the deterministic nonlocal graphon operator, as is detailed in Theorem~\ref{thm:Random}.  

As was pointed out in Section~\ref{sec:TuringonW}, ring networks have the advantageous property that their graphons can be represented in terms of Fourier series and the graphon eigenvalues and eigenvectors are simply derived from the Fourier coefficients and corresponding Fourier basis functions.  Since $W$ is translationally invariant it holds that all eigenvalues aside from the homogeneous one generically have algebraic and geometric multiplicity two.  In creating a random graph from the graphon this translational invariance is broken and so we expect generically that all eigenvalues of the matrix $\gL_r^N$ are simple.  

Our main result is the following, which is the random graph analog of Theorem~\ref{thm:graphonsinglemode}.

\begin{theorem}\label{thm:mainrandom}
{\color{black} Let  $W(x,y)=R(|x-y|)$ be an almost everywhere continuous ring graphon which can be represented as $W(x,y)=\sum_k c_k e^{2\pi \mbi k (x-y)}$.  Assume that there exists a $k^*$ such that $\mu:=c_{k^*}-c_0$ is an isolated eigenvalue of $\L$ with minimal algebraic multiplicity of two and $c_{2k^*}\neq 0$ so that $c_{2k^*}-c_0$ is an eigenvalue of finite (even) algebraic multiplicity $m$.  Then, there exists $\bar{\delta}:=\bar{\delta}(\mu,c_{2k^*}-c_0)>0$ such that for all $0<\delta<\bar{\delta}$ and any  $\gamma\in\left(0,\frac{1}{2}\right)$ there exists a $\bar{N}:=\bar{N}(\delta,\gamma)$, and a $C(\gamma)>0$ 
such that for any $N>\bar{N}$ the following are true with probability at least $1-2Ne^{-CN^{2\gamma}}$:
\begin{enumerate}
\item There exists an eigenvalue $\lambda_1(\gL_r^N)$ satisfying 
\be \left|\frac{\lambda_{1} (\gL_r^N)}{N}-(c_{k^*}-c_0)\right| <\delta. \ee
\item Fix $\kappa=\lambda_1$. Then for $|\e|$ sufficiently small there exists $\Omega \in \mathbb{R}$ such that \eqref{eq:main} has a nontrivial equilibrium solution which may be expanded as 
\be 
	u_j=\sqrt{2}z^*(\e,\delta) \cos\left(2\pi k^* \left(\frac{j-1}{N}-\Omega\right)\right)+\O(\e,\delta^2), 
\ee
for some $z^*(\e,\delta)$ given in \eqref{eq:bifsol}.  {\color{black} The bifurcating solution is stable if $\e>0$ and unstable (a saddle fixed point) if $\e<0$}.
\end{enumerate}}
\end{theorem}

\begin{rmk}Prior to providing the proof, we comment on the condition that $c_{2k^*} \neq 0$. If we had $c_{2k^*} = 0$, then the eigenvalue $c_{2k^*}-c_0$ lies at the accumulation point of the spectrum of $\mathcal{L_W}$. We would still be able to perform a center manifold reduction and obtain an equation analogous to \eqref{eq:genfirstreduced} below, however, since Theorem~\ref{thm:Random} does not provide good control over the difference between the eigenvalues of $\gL_r^N$ and those near the accumulation point for the spectrum of $\mathcal{L}(W)$, we do not have the means to predict the cubic coefficient on the one-dimensional center manifold based upon properties of the graphon alone. Therefore, this case is omitted since our analysis cannot be applied in this situation.
\end{rmk}


\begin{proof} 

{\color{black} Let $\gamma\in \left(0,\frac{1}{2}\right)$.  By Theorem~\ref{thm:Random} there exists a $\delta_0(\mu)$ such that for any $0<\delta<\delta_0(\mu)$ there exists a $N_0(\mu,\delta,\gamma)\geq 1$ such that with probability at least $1-2Ne^{-CN^{2\gamma}}$ there exists two eigenvalues $\lambda_{1,2}(\gL_r^N)$ satisfying 
\be 
	\left|\frac{\lambda_{1,2} (\gL_r^N)}{N}-(c_{k^*}-c_0)\right| <\delta, \label{eq:rgevals1to2} 
\ee
Similarly, since $c_{2k^*}-c_0$ is also an isolated eigenvalue of $\mathcal{L}(W)$, by Theorem~\ref{thm:Random} there exists a $\delta_0(c_{2k^*}-c_0)$ such that for any $0<\delta<\delta_0(c_{2k^*}-c_0)$ there exists a $N_0(c_{2k^*}-c_0,\delta,\gamma)\geq 1$ such that with probability at least $1-2Ne^{-CN^{2\gamma}}$ there exists $m$ eigenvalues satisfying 
\be 
	\left|\frac{\lambda_{3,4,\dots,2+m} (\gL_r^N)}{N}-(c_{2k^*}-c_0)\right| <\delta.  \label{eq:rgevals3tom} 
\ee
Set $\bar{\delta}=\min\{ \delta_0(\mu),\delta_0(c_{2k^*}-c_0)\}$ so that (\ref{eq:rgevals1to2}) and (\ref{eq:rgevals3tom}) hold for any $N\geq \bar{N}=\max\{ N_0(\mu,\delta,\gamma), N_0(c_{2k^*}-c_0,\delta,\gamma)\}$, with probability at least $1-2Ne^{-CN^{2\gamma}}$.

Additionally, Theorem~\ref{thm:Random} guarantees that the associated  eigenvectors for $\gL_r^N$ are $\O(\delta)$-close to the eigenvectors of the deterministic graph $\mathrm{L}_d$.  In turn, by Corollary~\ref{cor:CloseGraph} we have that the deterministic eigenvectors, by perhaps taking $\bar{N}$ larger, are $\O(\delta)$ close to the Fourier eigenfunctions of the graphon.  Consequently, with probability at least $1-2Ne^{-CN^{2\gamma}}$  the orthogonal eigenvectors of $\gL_r^N$ satisfy 
\be \begin{split} \left| \v_{1,2}-a_{1,2} \bomega_{k^*}-\overline{a_{1,2}} \bomega_{-k^*} \right|&< \delta \\
\left| \v_{j}-a_{j} \bomega_{2k^*}-\overline{a_{j}} \bomega_{-2k^*}-\sum_{n=1}^{\frac{m}{2}-1}\left( b_{jn} \bomega_{k_n}+\overline{b_{jn}}\bomega_{k_n}\right) \right|&< \delta, \end{split} \label{eq:v1v2}
\ee	
for all $j=3,4,\dots,2+m$, for some coefficients $a_k\in\mathbb{C}$, $b_{jn}\in\mathbb{C}$ and  where $\bomega_k$ are the normalized discrete Fourier basis vector 
\be 
	\bomega_k=\frac{1}{\sqrt{N}}\left( e^{2\pi k\mbi x_0}, e^{2\pi k \mbi x_1},\dots, e^{2\pi k \mbi x_{N-1}}\right)^T.
\ee}

By orthonormality of the eigenvectors $\v_k$ we also have that $ |a_1|^2=\frac{1}{2}+\O(\delta)$, or equivalently $a_1=\frac{1}{\sqrt{2}}e^{-2\pi\mbi \Omega}+\O(\delta)$ for some $\Omega\in \mathbb{R}$.  As a final preliminary, recall that zero is always an eigenvalue of the random graph Laplacian, $\gL_r^N$, and the associated eigenvector is constant.  We denote this eigenvalue/eigenvector pair as $(\lambda_N,\v_N)=(0,\bomega_0)$.  

Let us momentarily assume that $m=2$ so that $c_{2k^*}-c_0$ is also an isolated eigenvalue of $\mathcal{L}(W)$ of (minimal) algebraic multiplicity two.  We will return to the case of $m>2$ below.  For $m=2$, with {\color{black} with probability at least  $1-2Ne^{-CN^{2\gamma}}$} we have 
\be\label{eq:v3v4} 
	\begin{split}
		&\left| \v_3-a_3 \bomega_{2k^*}-\overline{a_3} \bomega_{-2k^*} \right|< \delta  \\
		&\left| \v_4-a_4 \bomega_{2k^*}-\overline{a_4} \bomega_{-2k^*} \right|< \delta. 
	\end{split}
\ee
Let $\v_5,\v_6,\dots,\v_{N-1}$ denote the remaining $N-5$ eigenvectors of $\gL_r^N$, which exist and are orthonormal since $\gL_r^N$ is symmetric.   This completes the set-up.  We now will assume that estimates (\ref{eq:rgevals1to2}),   (\ref{eq:rgevals3tom}) and (\ref{eq:v3v4}) are true to avoid repeating probabilistic statements throughout the proof.

We now perform a center manifold reduction.  To begin, expand the solution in its basis of eigenvectors as 
\be
	\u(t) = \sum_{k=1}^N w_k(t,\e)\v_k. 
\ee
This diagonalizes the linear part of \eqref{eq:main} and converts it into the system of equations
\be \label{eq:randomdiag} 
	\frac{dw_k}{dt}= l_k w_k+\e w_k +rQ_k(\w)+C_k(\w), 
\ee
where we can write $l_1=0$, $l_2=-\delta^2 N^2 \rho_2$, and $l_k=-N^2\rho_k$, with $\rho_k>0$ for all $k>1$ since $c_{k^*} - c_0$ is an isolated eigenvalue of $\mathcal{L}(W)$ with algebraic multiplicity $2$.  {\color{black} The function $Q_k:\mathbb{R}^N\to\mathbb{R}$ describes the quadratic terms in the equation for $w_k$ while $C_k:\mathbb{R}^N\to\mathbb{R}$ are cubic terms.  }

The quadratic terms in \eqref{eq:randomdiag} can be computed by projecting $\u(t)\circ \u(t)$ onto $\v_k$, or via the formula
\be
Q_k(\w)= \biggl \langle \v_k, \left(\sum_i w_i(t) \v_i \right) \circ \left(\sum_j w_j(t) \v_j\right)  \biggr \rangle =      \sum_{k_1,k_2}w_{k_1}w_{k_2}  \langle \v_k, \v_{k_1} \circ \v_{k_2} \rangle, \label{eq:Qform}
\ee 
where $\langle \u, \v\rangle$ is the Hermitian inner product and $\circ$ denotes the Hadamard product.  

Consider $k=1$. We may use \eqref{eq:v1v2} and \eqref{eq:v3v4} to find
\be 
	\sqrt{N} Q_1(\w)=r\beta w_1^2 +2r w_1w_N+ 2rq_{113}w_1w_3+2rq_{114}w_1w_4+2rq_{123}w_2w_3+2rq_{124}w_2w_4+\delta \tilde{Q}_1(\w),
\ee
where owing to the orthogonality of the discrete Fourier basis vectors we have 
\be \label{eq:beta}
	\beta=\sqrt{N} \langle \v_1, \v_1 \circ \v_1\rangle =\sqrt{N}\langle a_1 \bomega_{k^*}+\overline{a_1} \bomega_{-k^*}, a_1^2\bomega_{2k^*} +2 a_1\overline{a_1}\bomega_0+\overline{a_1}^2\bomega_{-2k^*} \rangle+ \O(\delta) =\O(\delta), 
\ee
and 
\be 
	q_{113}=a_1^2\overline{a_3}+\overline{a_1}^2a_3, \quad  q_{114}=a_1^2\overline{a_4}+\overline{a_1}^2a_4, \quad q_{123}=a_2^2\overline{a_3}+\overline{a_2}^2a_3, \quad q_{124}=a_2^2\overline{a_4}+\overline{a_2}^2a_4.
\ee
The cubic terms can be computed by the formula
\be 
	C_k(\w)= \biggl \langle \v_k, \left(\sum_i w_i(t) \v_i \right) \circ \left(\sum_j w_j(t) \v_j\right) \circ \left(\sum_l w_l(t) \v_l\right) \biggr \rangle =      \sum_{k_1,k_2,k_3}w_{k_1}w_{k_2}w_{k_3}  \langle \v_k, \v_{k_1} \circ \v_{k_2}\circ \v_{k_3} \rangle. 
\ee
So, we have
\be
	\begin{split}
	 	\langle \v_1, \v_{1} \circ \v_{1}\circ \v_{1} \rangle &=\frac{1}{N}\langle a_1\bomega_{k^*}+\overline{a_1}\bomega_{-k^*}, a_1^3 \bomega_{3k^*}+3a_1^2 \overline{a_1} \bomega_{k^*} +3a_1 \overline{a_1}^2 \bomega_{-k^*} +\overline{a_1}^3 \bomega_{-3k^*}\rangle +\O(\delta)  \\
	&=\frac{1}{N}\langle a_1\bomega_{k^*}+\overline{a_1}\bomega_{-k^*}, 3a_1^2 \overline{a_1} \bomega_{k^*} +3a_1 \overline{a_1}^2 \bomega_{-k^*} \rangle +\O(\delta).
	\end{split}
\ee
 Using $|a_1|^2=\frac{1}{2}+\O(\delta)$ therefore implies  
\be
	C_1(\w)=-\frac{3b}{2N}w_1^3+ \delta \tilde{C}_1(\w,\delta).
\ee
We now have obtained the required expansions for the equation that governs $w_1$ in (\ref{eq:randomdiag}).  Since $\ell_1=0$ and $\ell_k<0$ for all other $k$ we then have the existence of a one-dimensional center manifold.  {\color{black} The center manifold can be described locally as a graph over the center eigenspace, parameterized by $w_1$ and $\varepsilon$.  Let $w_k=H_k(w_1,\e)$ denote this graph for each component $k\neq 1$.   Our first objective is to compute expansions for the graphs $H_k(w_1,\e)$.  We will focus our attention on the terms $k=3,4$ and $N$ and compute the leading order quadratic terms in their expansions.  After that, we compute the reduced dynamics on the one-dimensional center manifold by substituting these expansions into the differential equation for $w_1$.  We will see that quadratic expansions for the graphs $H_k(w_1,\e)$ will be sufficient to obtain the reduced dynamics to cubic order and this will be sufficient to characterize the bifurcation that occurs at $\e=0$.  }   

The quadratic expansions for $H_k(w_k,\e)$ are computed as follows.   Consider first the case of $w_N$.  Expanding (\ref{eq:randomdiag}) and  using (\ref{eq:Qform})  we can compute the quadratic term in the expansion of $H_N(w_1,\e)$ by a normal form transformation where we find $\zeta$ for which $w_N=\zeta w_1^2$ satisfies (to leading order)
\be \frac{dw_N}{dt}=l_N w_N+\frac{r}{\sqrt{N}}w_1^2. \ee
Thus, $\zeta=-\frac{r}{\sqrt{N}l_N}+\O(\delta)$.  This procedure yields expansions for $H_3$, $H_4$ and $H_N$ which end up being,
\be
	\begin{split}
		H_3(w_1,\e)&= -\frac{r \left(a_3\overline{a_1}^2+\overline{a_3}a_1^2\right)}{\sqrt{N}l_3}w_1^2+\O(w_1^3+\e w_1^2) +\O(\delta) \\
		H_4(w_1,\e)&= -\frac{r \left(a_4\overline{a_1}^2+\overline{a_4}a_1^2\right)}{\sqrt{N}l_4}w_1^2+\O(w_1^3+\e w_1^2)+\O(\delta) \\
		H_N(w_1,\e)&=-\frac{r}{\sqrt{N}l_N} w_1^2 +\O(w_1^3+\e w_1^2)+\O(\delta).
	\end{split}
\ee
With these expansions, we may now compute the reduced dynamics in the center manifold. We find
\be \label{eq:reducedrandom1}  
	\frac{dw_1}{dt}=\e w_1+\frac{r\beta}{\sqrt{N}} w_1^2-\frac{1}{N}\left(\frac{3b}{2}+\frac{2r^2}{l_N}+\frac{2r^2q_{113}^2}{l_3}+\frac{2r^2q_{114}^2}{l_4}\right)w_1^3+\O(w_1^4+\e w_1^3+\delta w_1^3+\delta^2w_1^2).
\ee
The cubic terms can be simplified by noting that orthogonality of the eigenvectors $\v_k$ implies that $a_k\overline{a}_j+ a_j\overline{a}_k= \O(\delta)$ for all $k\neq j$ while normality implies $a_k\overline{a}_k=\frac{1}{2}+\O(\delta)$.  Taken all together these facts imply that $a_3^2+a_4^2=\O(\delta)$ and hence
\be 
	q_{113}^2+q_{114}^2=a_1^4\left(\overline{a}_3^2+\overline{a}_4^2\right) +2|a_1|^4 \left(|a_3|^2+|a_4|^2\right) +\overline{a}_1^4\left(a_3^2+a_4^2\right) =\frac{1}{2}+\O(\delta).
\ee
Finally, since $|l_3-l_4|=\O(\delta)$ we have then reduced \eqref{eq:reducedrandom1} to 
\be \label{eq:reducedrandom}
	\frac{dw_1}{dt}=\e w_1+\frac{r \beta}{\sqrt{N}} w_1^2-\frac{1}{N}\left(\frac{3b}{2}+\frac{2r^2}{l_N}+\frac{r^2}{l_3} \right)w_1^3+\O(w_1^4+\e w_1^3+\delta w_1^3+\delta^2w_1^2), 
\ee
which constitutes the leading order expansion of the center manifold of \eqref{eq:main} near $(u,\e) = (0,0)$. 

To obtain steady-state solutions on the center manifold, begin by letting $w_1=\sqrt{N}z_1$. Then \eqref{eq:reducedrandom} is transformed to
\be\label{eq:reducedrandomz2} 
	\frac{dz_1}{dt}=\e z_1+r \beta z_1^2-\left(\frac{3b}{2}+\frac{2r^2}{l_N}+\frac{r^2}{l_3} \right)z_1^3+\O(z_1^4+\e z_1^3+\delta z_1^3),  
\ee
where we recall from \eqref{eq:beta} that $\beta = \mathcal{O}(\delta)$. Applying the implicit function theorem we find that for $\e$ and $\delta$ sufficiently small there exists a pair of equilibrium points
\be  \label{eq:z1def}
	z_1^\pm (\e,\delta)= \frac{\beta r}{2\Gamma_r}\pm \frac{1}{2\Gamma_r}\sqrt{\beta^2r^2 +4\Gamma_r \e}+\O(\e,\delta^2),  
\ee 
where 
\be
	\Gamma_r=\frac{3b}{2}+\frac{2r^2}{l_N}+\frac{r^2}{l_3}
\ee
is the cubic coefficient on the center manifold.  For any $\delta\neq 0$ the bifurcation occuring near the zero state is transcritical, while if $\delta$ is close to zero it is accompanied by a saddle node bifurcation occurring at 
\be 
	\e_{SN}\approx- \frac{\beta^2 r^2}{4\Gamma_r}. 
\ee
The sign of $\Gamma_r$ determines whether the equilibrium $z_1^\pm(\e,\delta)$ exist for $\e>\e_{SN}$ ($\Gamma_r>0$) or for $\e<\e_{SN}$ ($\Gamma_r<0$).  We now revert to the original coordinates of (\ref{eq:main}).

{\color{black} Reverting coordinates back to those of (\ref{eq:main}),  we have the existence of a bifurcating solution of the form
\be 
	u_j=\sqrt{2} z^*(\e,\delta) \cos\left(2\pi k^* \left(\frac{j-1}{N}+\Omega\right)\right)+\O(\e,\delta^2),  \label{eq:bifsol}
\ee
for  $\Omega$ defined above where $z^*(\e,\delta)=z_1^-(\e,\delta)$ if $\beta r>0$ and $z^*(\e,\delta)=z_1^+(\e,\delta)$ otherwise.   }

The preceding analysis assumed that $c_{2k^*}-c_0$ is an isolated eigenvalue of $\mathcal{L}(W)$ of minimal multiplicity two.  We now consider the situation where this eigenvalue remains isolated but has multiplicity $m>2$ so that the expansions (\ref{eq:v1v2}) hold, i.e  
\be 
	\v_j=a_j \bomega_{2k^*}+\overline{a_j} \bomega_{-2k^*} +\sum_{n=1}^{\frac{m}{2}-1}\left( b_{jn} \bomega_{k_n}+\overline{b_{jn}}\bomega_{k_n}\right)  +\delta p_j, \quad 3\leq j\leq 3+m,  
\ee
for some nonzero $k_n\neq k^*$ {\color{black} and again this holds with probability at least $1-2Ne^{-CN^{2\gamma}}$}.  Orthonormality of the eigenvectors implies that 
\be\label{eq:aident}  
	\sum_{j=3}^{3+m} a_j\overline{a_j}=1+\O(\delta) , \quad  \sum_{j=3}^{3+m}a_j^2=\O(\delta),
\ee
and so we can generalize \eqref{eq:randomdiag} to this case and obtain
\be 
	\sqrt{N} Q_1(\w)=r\beta w_1^2 +2r w_1w_N +2r\sum_{j=3}^{3+m} \left(q_{11j}w_1w_j+q_{12j}w_2w_j \right)+\delta \tilde{Q}_1(\w), 
\ee
where $q_{11j}=a_j \overline{a_1}^2+\overline{a_j}a_1^2$.  Repeating the analysis above, we find expansions for the graph defining the center manifold as 
\be
	\begin{split}
		H_j(w_1,\e) &= -\frac{r q_{11j}}{\sqrt{N} l_j}w_1^2 +\O(w_1^3+\e w_1^2), \quad 3\leq j\leq 3+m \\
		H_N(w_1,\e) &= -\frac{r}{\sqrt{N} l_N}w_1^2 +\O(w_1^3+\e w_1^2). 
	\end{split}
\ee
We then compute the reduced equation on the center manifold to obtain, in analogy with \eqref{eq:reducedrandom1}, 
\be \label{eq:genfirstreduced}
	\frac{dw_1}{dt}=\e w_1+\frac{r\beta}{\sqrt{N}} w_1^2-\frac{1}{N}\left(\frac{3b}{2}+\frac{2r^2}{l_N}+\sum_{j=3}^{3+m} \frac{2r^2q_{11j}^2}{l_j}\right)w_1^3 +\O(w_1^4+\e w_1^3+\delta w_1^3+\delta^2w_1^2).  
\ee
Employing (\ref{eq:aident}) this reduces to 
\[ \frac{dw_1}{dt}=\e w_1+\frac{r \beta}{\sqrt{N}} w_1^2-\frac{1}{N}\left(\frac{3b}{2}+\frac{2r^2}{l_N}+\frac{r^2}{l_3} \right)w_1^3+\O(w_1^4+\e w_1^3+\delta w_1^3+\delta^2w_1^2).  \]
which has the identical form to the previous case seen in \eqref{eq:reducedrandom}. This completes the proof.
\end{proof} 

\begin{rmk} A direct link between the bifurcation equation obtained for the random graph in \eqref{eq:reducedrandomz2} and that of the non-local graphon model \eqref{eq:scalarz} could be obtained by an $N$-dependent rescaling of coefficients.  If we re-scale parameters by $\e=\tilde{\e} N^2$, $r=\tilde{r}N^2$,  $b=\tilde{b}N^2$, re-scale $z_1=\sqrt{2}\tilde{z}_1$ and re-scale the dependent variable by $t=\tau N^{-2}$ then \eqref{eq:reducedrandomz2} is transformed to 
\be \label{eq:reducedrandomz3}
	\frac{d\tilde{z}_1}{d\tau}=\tilde{\e} \tilde{z}_1+\sqrt{2}\tilde{r} \beta \tilde{z}_1^2-\left(3\tilde{b}+\frac{4\tilde{r}^2}{\ell_{0}}+\frac{2\tilde{r}^2}{\ell_{2k^*}} \right)\tilde{z}_1^3+\O(\tilde{z}_1^4+\e \tilde{z}_1^3+\delta \tilde{z}_1^3).  
\ee
\end{rmk}

We now present the analogous result to Theorem~\ref{thm:mainrandom} for the case when the graphon Laplacian has bifurcating modes that are in $2:1$ resonance.  We show that these bifurcations manifest as transcritical bifurcations in \eqref{eq:main} which stem from resonant graphon eigenvalues.

\begin{theorem}\label{thm:randomresonance} 
{\color{black} Let  $W(x,y)=R(|x-y|)$ be an almost everywhere continuous ring graphon which can be represented as $W(x,y)=\sum_k c_k e^{2\pi \mbi k (x-y)}$.  Assume that there exists a $k^*$ such that $c_{k^*}=c_{2k_*}$ and $\mu:=c_{k^*}-c_0$ is an isolated eigenvalue of $\L$ with minimal algebraic multiplicity of four. Then, there exists $\bar{\delta}>0$ such that for all $0<\delta<\bar{\delta}$ and any  $\gamma\in\left(0,\frac{1}{2}\right)$ there exists a $\bar{N}(\delta,\gamma)$, a $C(\gamma)>0$ 
such that for any $N>\bar{N}$ the following is true with probability exceeding $1-2Ne^{-CN^{2\gamma}}$:
\begin{enumerate}
\item There exists an eigenvalue $\lambda_1(\gL_r^N)$ satisfying 
\be \left|\frac{\lambda_{1} (\gL_r^N)}{N}-(c_{k^*}-c_0)\right| <\delta. \ee
\item If $\kappa=\lambda_1$ then for $|\e|$ sufficiently small, there exists $\Omega_{1,2} \in \mathbb{R}$ and $\eta_{1,2}>0$ such that \eqref{eq:main} has a nontrivial equilibrium solution which may be expanded as 
\be 
	u_j=\e \eta_1 \cos\left(2\pi k^* \left(\frac{j-1}{N}-\Omega_1\right)\right)+\e\eta_2 \cos\left(4\pi k^* \left(\frac{j-1}{N}-\Omega_2\right)\right)+\O(\e^2,\delta^2).
\ee
\end{enumerate} }
\end{theorem}

\begin{proof}
The proof proceeds similarly to that of Theorem~\ref{thm:mainrandom} and so we only highlight the differences here. Suppose the same-set up as in Theorem~\ref{thm:mainrandom}, except we assume that the graphon eigenvalue has multiplicity four and the bifurcating modes are $e^{2\pi\mbi k^* x}$ and $e^{4\pi\mbi k^*x}$.  From Theorem~\ref{thm:Random} we have that for any $\delta > 0$ taken sufficiently small {\color{black} and with probability at least $1-2Ne^{-CN^{2\gamma}}$}  there exists an eigenvalue, $\lambda_1$, of $\gL_r^N$  such that
\be 
	\left|\frac{\lambda_{1} (\mathrm{L}_r)}{N}-(c_{k^*}-c_0)\right| <\delta. 
\ee	
{\color{black} Moreover, with probability at least  $1-2Ne^{-CN^{2\gamma}}$ the corresponding eigenvector satisfies,
\be 
	\left| \v_1-a_1 \bomega_{k^*}-\overline{a_1} \bomega_{-k^*}  -a_2 \bomega_{2k^*}-\overline{a_2} \bomega_{-2k^*} \right|<\delta, 
\ee
for some coefficients $a_{1,2}\in\mathbb{C}$.} We can then follow the center manifold reduction in Theorem~\ref{thm:mainrandom}, but now we note that  the quadratic self interaction term is non-zero since
\be
	\begin{split} 
		\sqrt{N} \v_1\circ \v_1&=a_2^2 \bomega_{4k^*}+2a_1a_2 \bomega_{3k^*}+a_1^2\bomega_{2k^*}+2a_2\overline{a_1} \bomega_{k^*} +\left(2a_1\overline{a_1}+2a_2\overline{a_2}\right)\bomega_0+2a_1\overline{a_2} \bomega_{-k^*}  \\
		&\quad \quad +\overline{a_1}^2 \bomega_{-2k^*}+2\overline{a_1}\overline{a_2} \bomega_{-3k^*}+\overline{a_2}^2 \bomega_{-4k^*}+\O(\delta),
	\end{split}
\ee
giving that 
\be 
	\sqrt{N} \langle \v_1, \v_1 \circ \v_1\rangle =3a_1^2\overline{a_2}+3a_2 \overline{a_1}^2+\O(\delta). 
\ee
In this case the reduced flow on the center manifold is then described by the equation
\be 
	\frac{dw_1}{dt}=\e w_1+\frac{r\beta}{\sqrt{N}}w_1^2 +\O(3), 
\ee
where $\beta=3a_1^2\overline{a_2}+3a_2 \overline{a_1}^2+\O(\delta)$. If $a_1\neq 0$ and $a_2\neq 0$ then the quadratic coefficient is non-zero and $\O(1)$ in $\delta$. The remainder of the proof is now straightforward and follows as in the previous result.
\end{proof}


\section{Examples and Numerical Simulations}\label{sec:numerics}

We conclude with a numerical investigation of (\ref{eq:main}) to illustrate our main results and to demonstrate how different the phenomena may be in the case where the hypothesis of our main result are not satisfied.


\subsection{Example: small-world network} 

Consider the small-world graphon $W(x,y)$. Expanding the graphon as a Fourier Series
\be
	W(x,y)=\sum_{k\in \mathbb{Z}} c_k e^{2\pi\mbi k(x-y) }, 
\ee
we compute that 
\be 
	c_k= \bigg(\frac{p-q}{\pi k}\bigg) \sin (2\pi k \alpha), \quad c_0=2\alpha p+(1-2\alpha) q.
\ee
Hence, the graphon Laplacian eigenvalues are given by 
\be \label{eq:WSevals}
	\lambda_k=c_k-c_0= \bigg(\frac{p-q}{\pi k}\bigg) \sin (2\pi k \alpha)- 2\alpha p-(1-2\alpha) q. 
\ee

\begin{figure}
    \centering
     \subfigure{\includegraphics[width=0.43\textwidth]{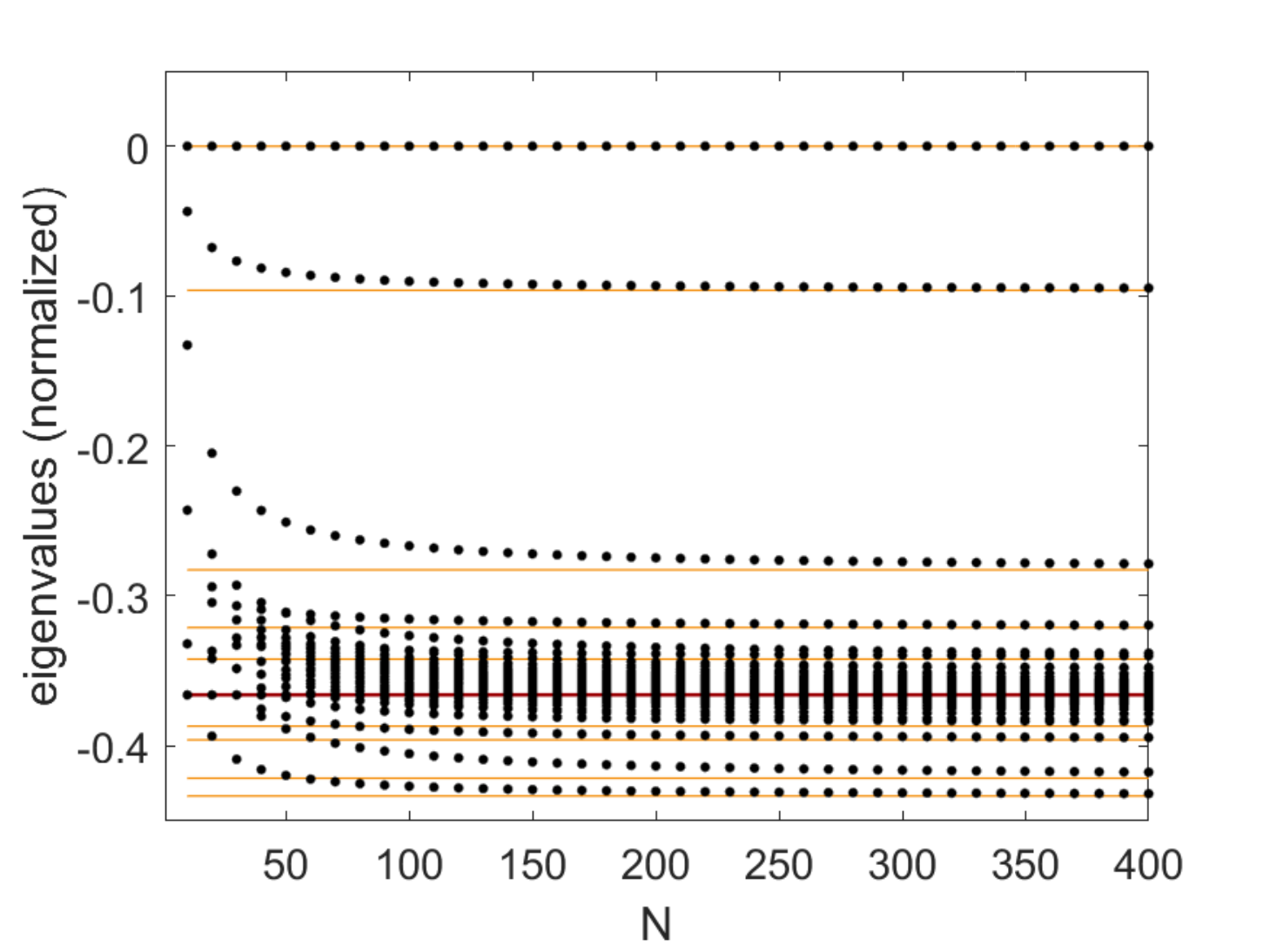}}
 \subfigure{\includegraphics[width=0.43\textwidth]{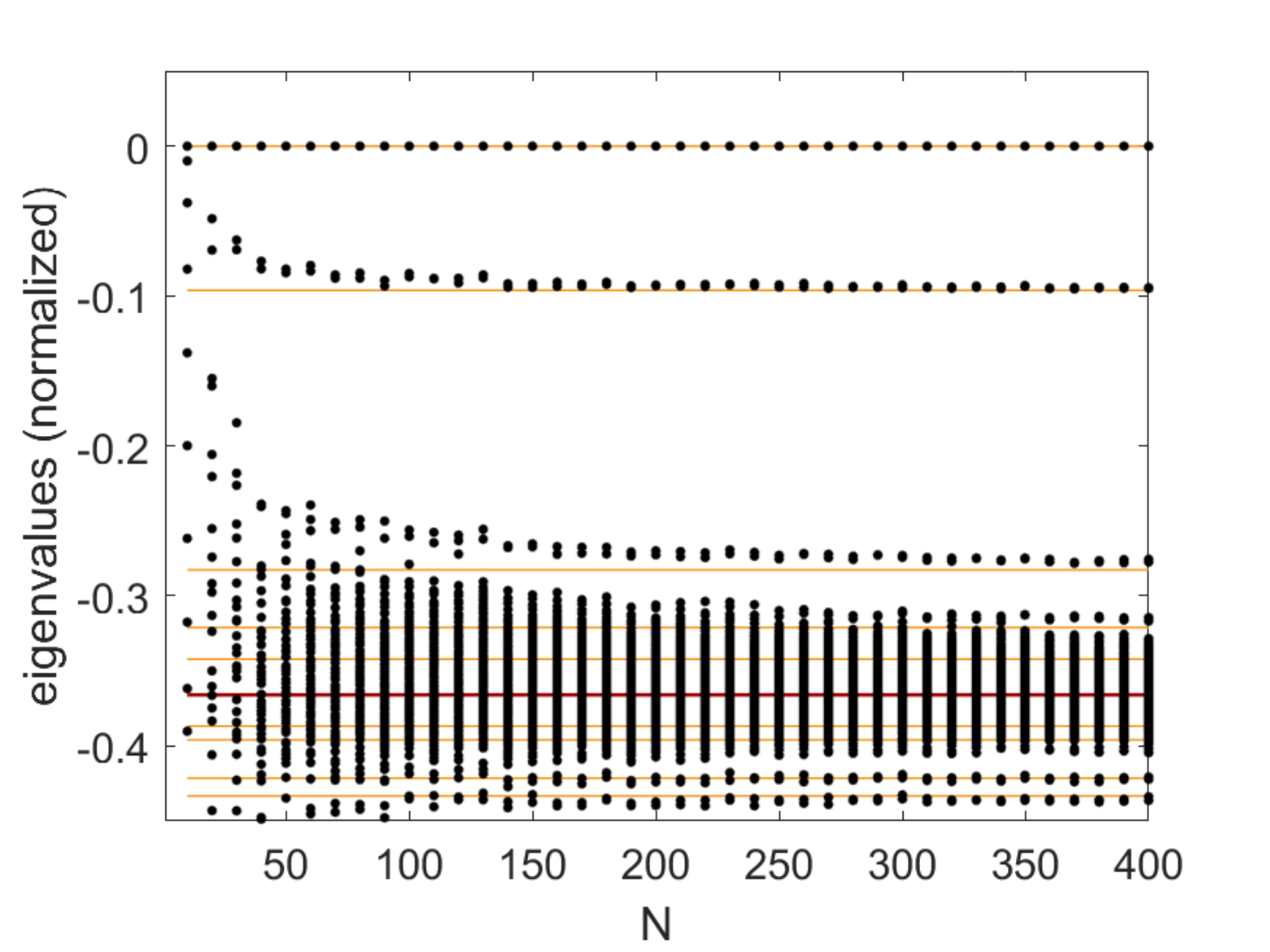}}
   \caption{Convergence of the (normalized) eigenvalues of the deterministic (left) and random (right) graph Laplacians for a small-world network with parameters $(p,q,\alpha) = (0.90,0.01,0.20)$.  In both cases we show convergence of the eigenvalues for the graph Laplacians as the number of nodes is increased from $N=10$ to $N=400$. In both figures the horizontal lines are the eigenvalues of the graphon defined in \eqref{eq:WSevals} for $0\leq k\leq 10$. The darker line represents the accumulation point of the spectrum at $\lambda_0=c_0 = -0.366$.}
    \label{fig:evalconv}
\end{figure}

In Figures~\ref{fig:evalconv}-\ref{fig:bifbad} we present our numerical investigations of \eqref{eq:main} for small-world graphs with parameters $(p,q,\alpha) = (0.90,0.01,0.20)$ and system parameters $(r,b) = (1,1)$. First, in Figure~\ref{fig:evalconv} we provide the eigenvalues of $\mathrm{L}_d^N$ and realizations of $\mathrm{L}_r^N$ for graphs of size $N=10$ to $N=400$. Vertical lines represent the eigenvalues $\lambda_k$ of the ring graphon Laplacian used to generate $\gL_d^N$ and $\gL_r^N$. In both cases we see pairs of eigenvalues converging (after rescaling by $N$) to the eigenvalues of the graphon Laplacian, confirming the analysis in Theorem~\ref{thm:Random}. Note that the eigenvalue near $-0.1$ of the graphon Laplacian is isolated and when $N = 400$ the corresponding pair of eigenvalues of $\mathrm{L}_r^N$ are close to this isolated eigenvalue, meaning that Theorem~\ref{thm:mainrandom} can be applied here. So, in Figure~\ref{fig:bifgood} we take a realization of the random graph with $N = 400$ and numerically compute the pattern-forming Turing bifurcation curve in a neighborhood of $(u,\e)=(0,0)$ using numerical continuation.  As predicted by our analysis, the bifurcation is weakly transcritical and resembles the super-critical pitchfork bifurcation predicted by the graphon analysis.  We also plot the solution computed along this branch and observe that it closely resembles a translate of the bifurcating Fourier mode predicted by the graphon. To contrast this, we perform the same calculation for the fiftieth largest eigenvalue in Figure~\ref{fig:bifbad} using the same realization of the random graph. Here this eigenvalue is close to the accumulation point at $c_0 = -0.3666$ (after normalizing by $N = 400$) and is not sufficiently isolated to apply our analytical result.  This is consistent with the numerical simulations, where the bifurcation is seen to be strongly transcritical and the bifurcating solution lacks any clear spatial structure.  

\begin{figure}
    \centering
     \subfigure{\includegraphics[width=0.43\textwidth]{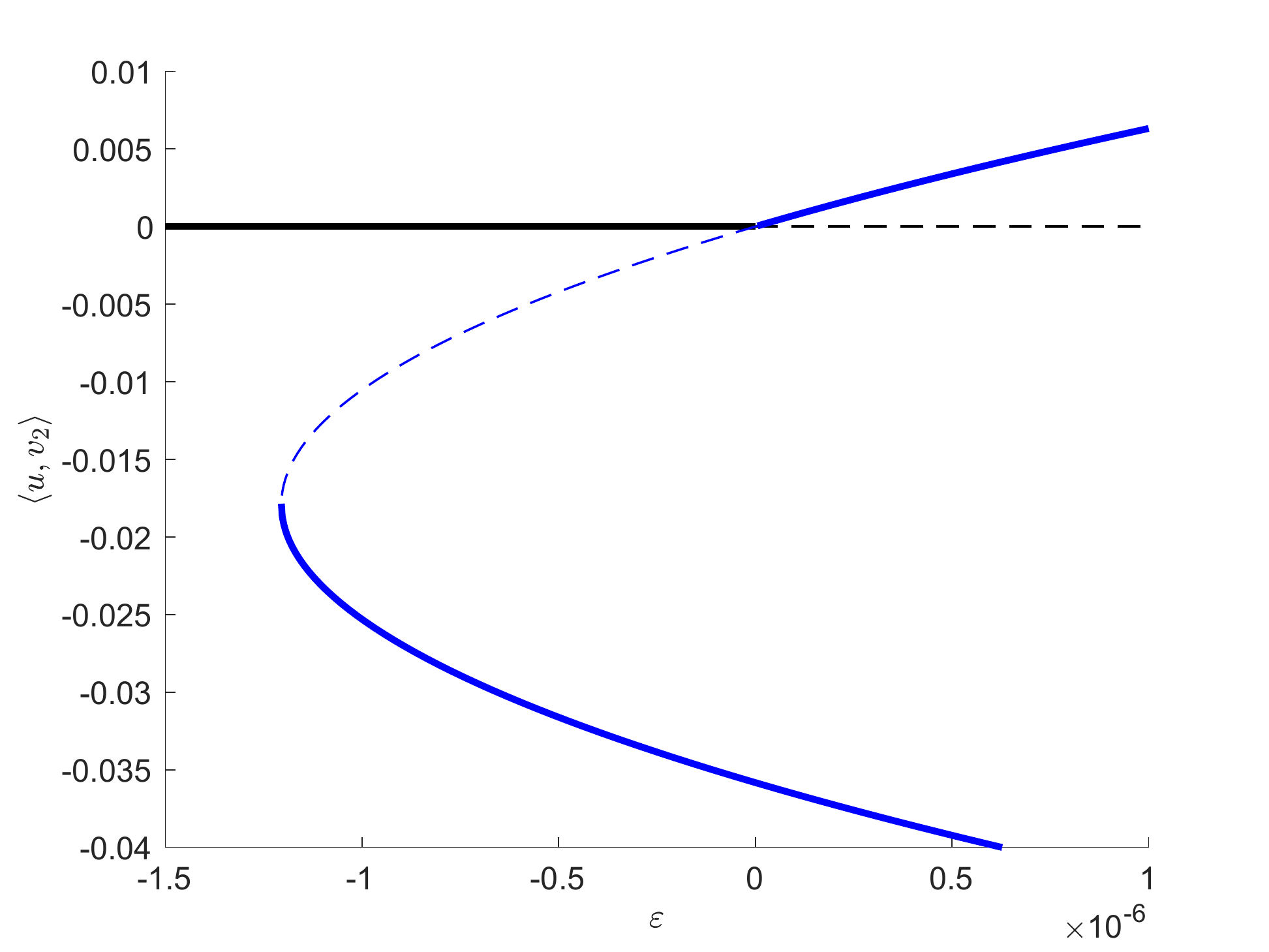}}
 \subfigure{\includegraphics[width=0.43\textwidth]{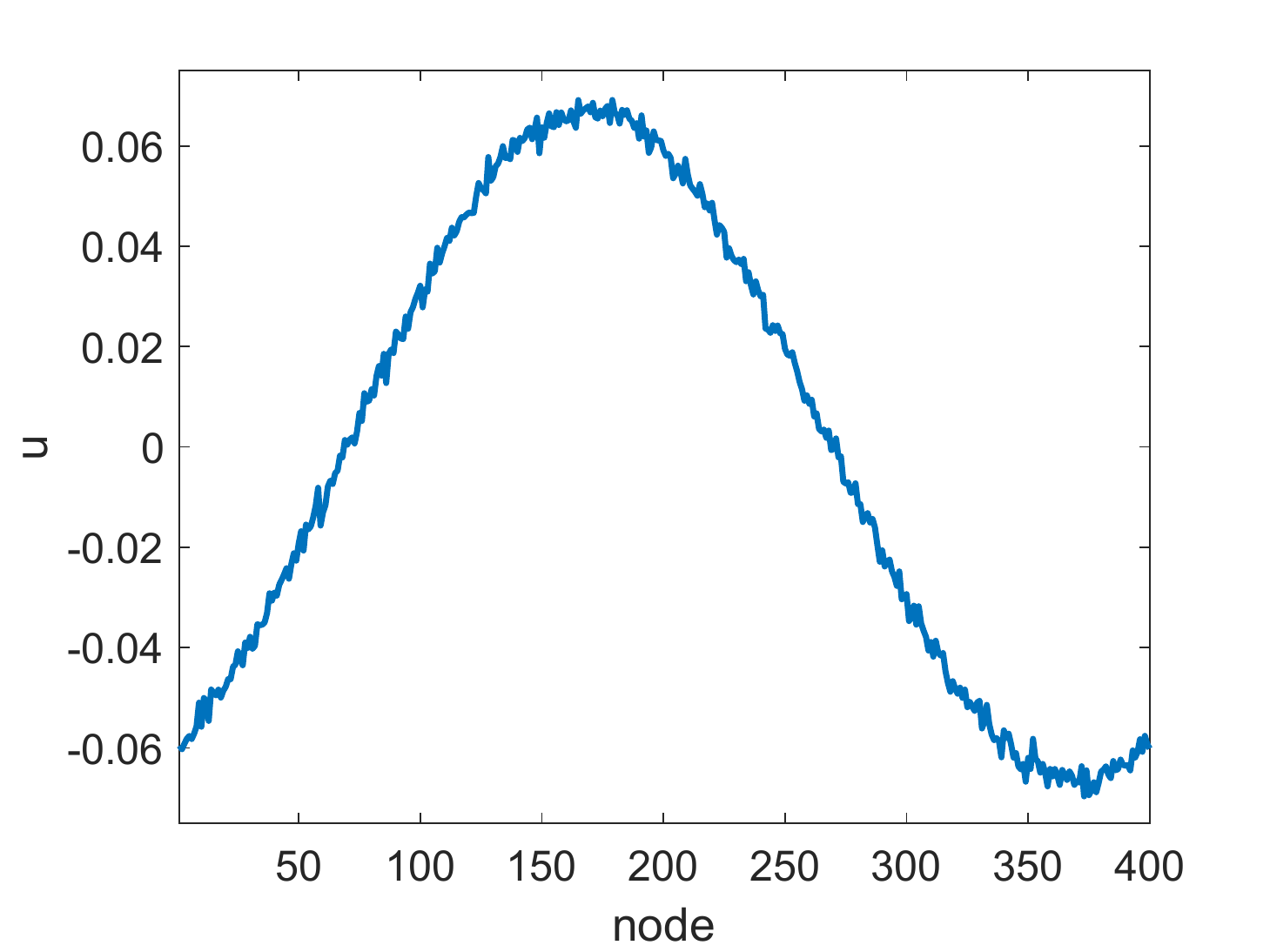}}
   \caption{On the left is the bifurcation diagram near $(u,\e) = (0,0)$ for the discrete Swift--Hohenberg equation with $(\kappa,r,b) = (-37.7255,1,1)$ on a random small-world network of size $N=400$ with parameters $(p,q,\alpha) = (0.90,0.01,0.20)$. Stable solutions are denoted by a solid line, while those that are unstable are given by a dashed line. Here $\kappa$ corresponds to the largest (non-zero) eigenvalue of $\mathrm{L}_r^N$.  On the right is the bifurcating solution with $\e=0.0029$ which closely resembles a translate of the Fourier mode $\cos\left(\frac{2\pi n}{400}\right)$.  We note that the numerically computed quadratic interaction coefficient in this example is $\v^T \v\circ \v=0.000515$, where $\v$ is the eigenvector corresponding to the critical eigenvalue.}
    \label{fig:bifgood}
\end{figure}

\begin{figure}
    \centering
     \subfigure{\includegraphics[width=0.43\textwidth]{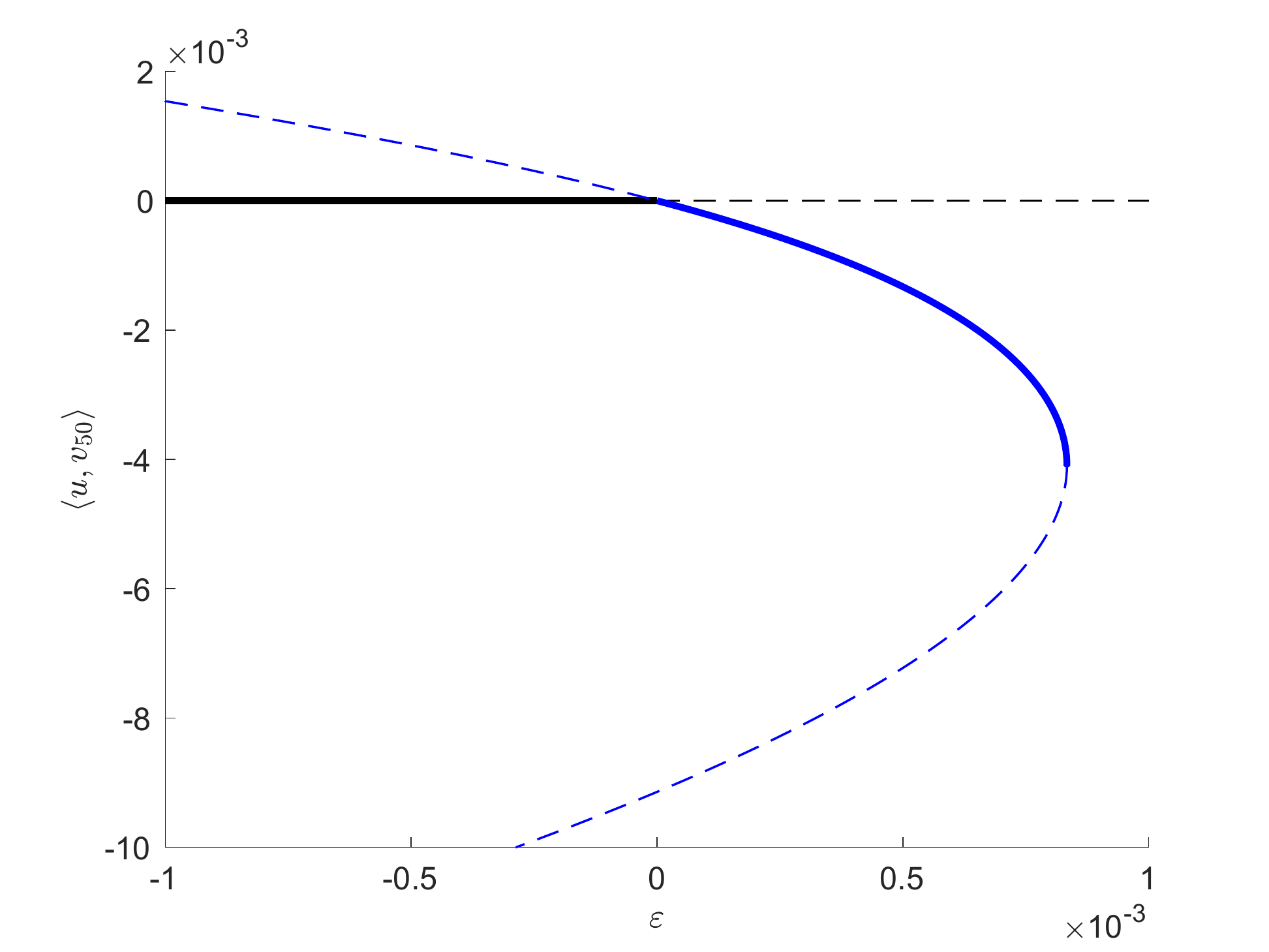}}
 \subfigure{\includegraphics[width=0.43\textwidth]{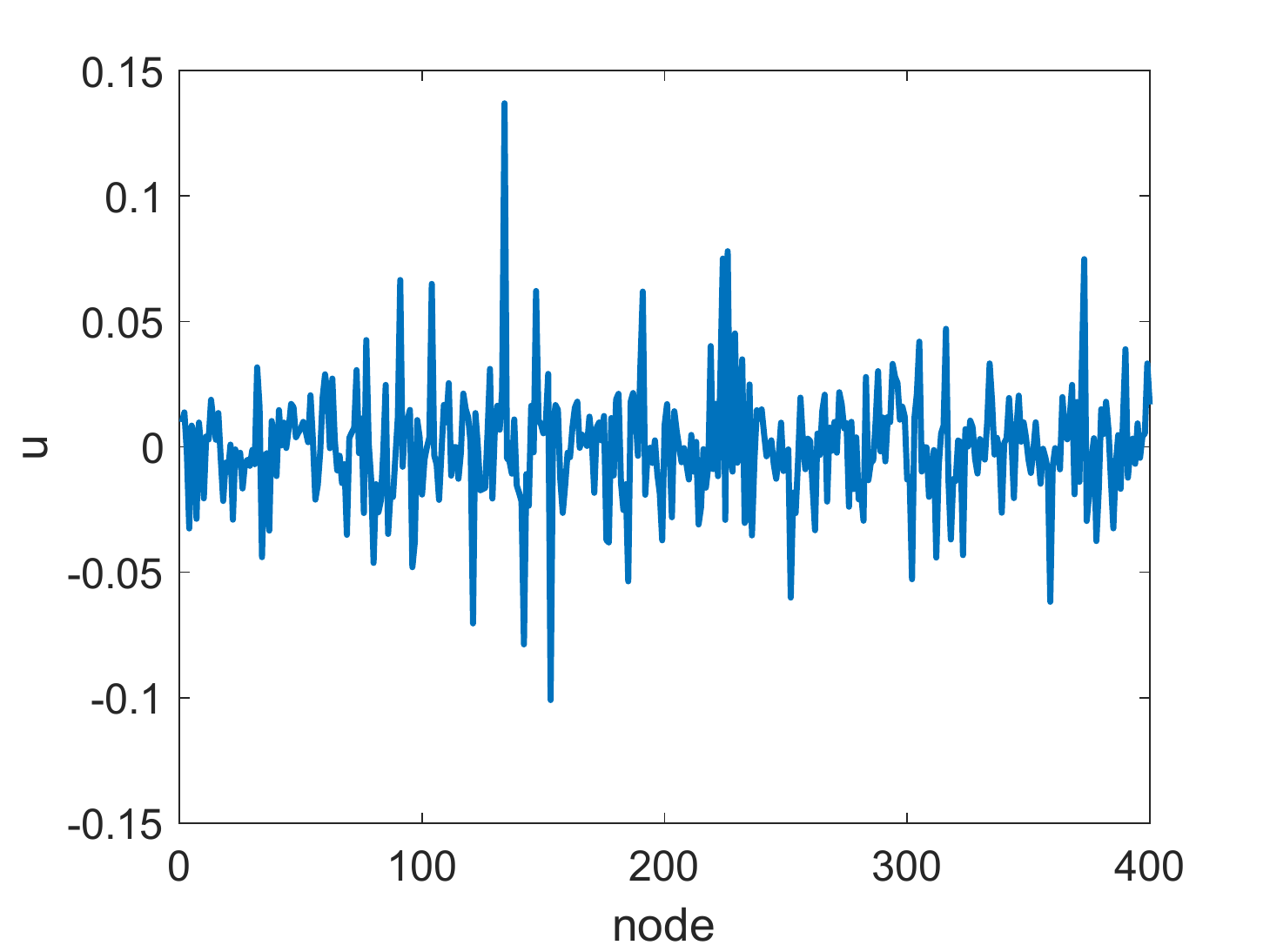}}
   \caption{On the left is the bifurcation diagram near $(u,\e) = (0,0)$ for the discrete Swift--Hohenberg equation with $(\kappa,r,b) = (-137.2908,1,1)$ on the same small-world network as in Figure~\ref{fig:bifgood}. Here we take $\kappa$ to be the fiftieth largest eigenvalue of $\mathrm{L}_r^N$, which is not sufficiently isolated for the results of Section~\ref{sec:Turinggraph} to apply.   Contrasting the axis scales in this figure with those of Figure~\ref{fig:bifgood}, we see that the bifurcation is strongly transcritical in this case as compared to the example provided in Figure~\ref{fig:bifgood}. The numerically computed quadratic interaction coefficient here is $\v^T \v\circ \v=-0.0157$, where $\v$ is the eigenvector corresponding to the critical eigenvalue. This is a thirty fold increase over the coefficient obtained from the isolated principal eigenvalue. On the right we provide the bifurcating solution for $\e=-0.0012$ where we observe a lack of spatial structure.}
    \label{fig:bifbad}
\end{figure}


\subsection{Example: a graph with $2:1$ resonance}

To illustrate the results of Theorem~\ref{thm:randomresonance}, we consider the toy graphon model
\be
	W(x,y) = \frac{1}{2} + \frac{1}{4}\cos(2\pi(x - y)) + \frac{1}{4}\cos(4\pi(x-y)). \label{eq:Wres} 
\ee
The above graphon has $c_0 = \frac 1 2$, $c_1 = c_2 = \frac{1}{8}$, and all other $c_k = 0$ in its Fourier expansion, and so the corresponding graphon Laplacian has eigenvalues given by $0, - \frac 3 8, -\frac 1 2$ . Moreover, the eigenvalue $-\frac 3 8$ is isolated and comes in a $2:1$ resonance since $c_1 = c_2$.  For realizations of the random graph with $N=400$ nodes we find, as expected, that there are four eigenvalues near $-\frac{3N}{8} = -150$.  The results of Theorem~\ref{thm:randomresonance} shows that with high probability this configuration leads to a transcritical bifurcation. We further confirm these results with Figure~\ref{fig:resonance}, which shows a transcritical bifurcation with the bifurcating solution approximately a linear superposition of the involved Fourier modes.

\begin{figure} 
    \centering
     \subfigure{\includegraphics[width=0.43\textwidth]{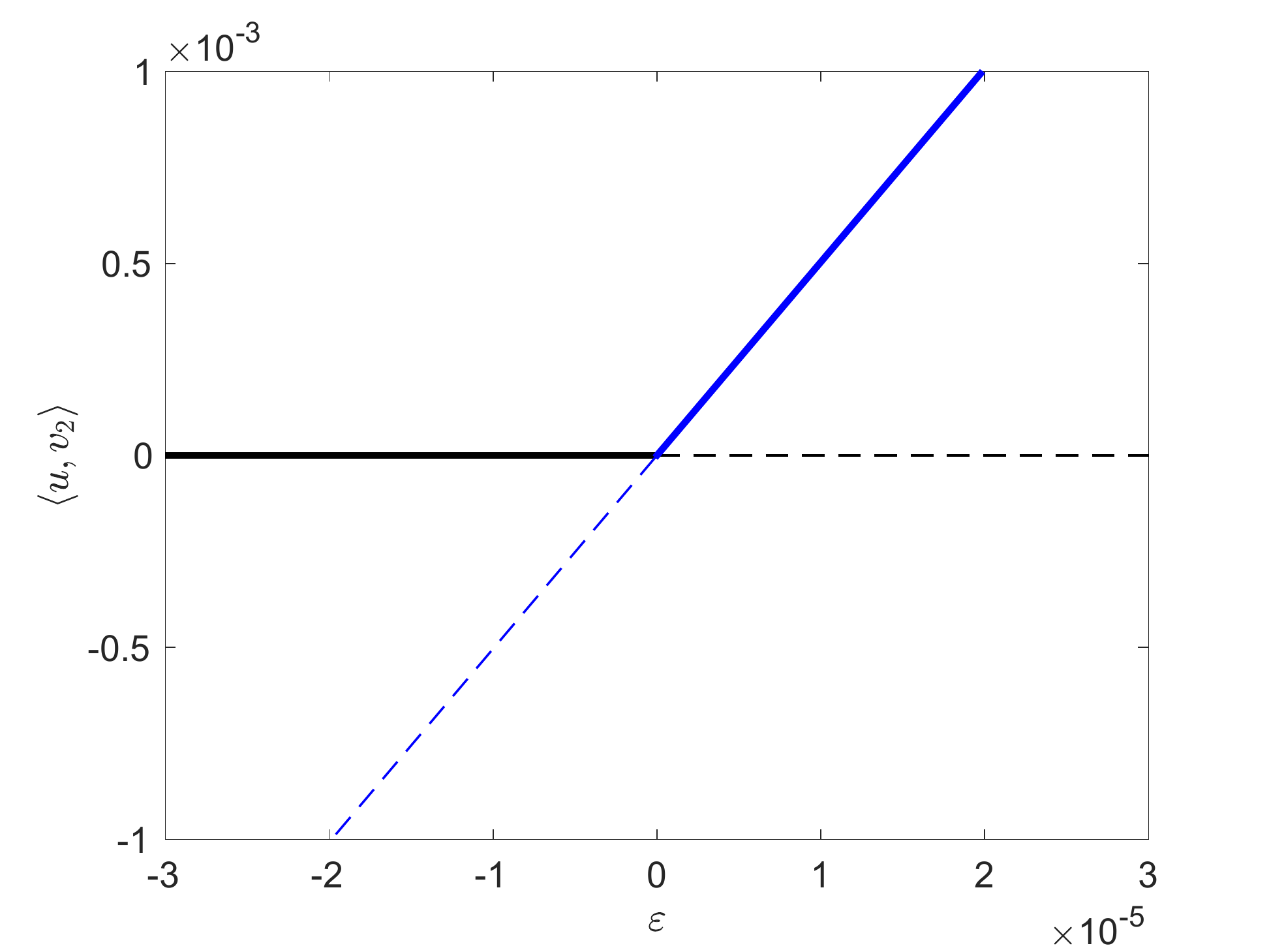}}
 \subfigure{\includegraphics[width=0.43\textwidth]{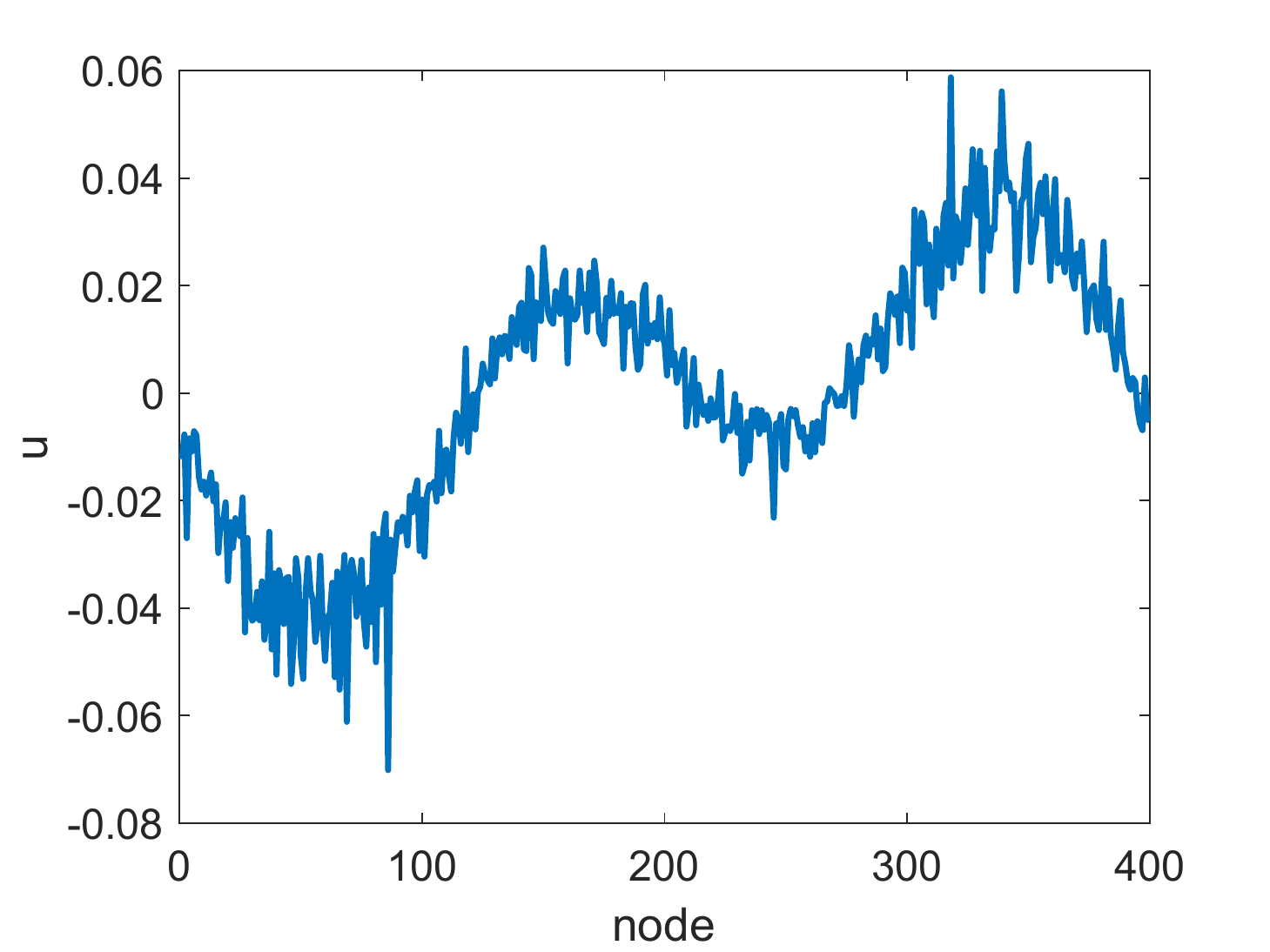}}
   \caption{A transcritical bifurcation occurring due to the presence of a $2:1$ resonance in the graphon \eqref{eq:Wres}. On the left is the bifurcation diagram near $(u,\e) = (0,0)$ for the discrete Swift--Hohenberg equation with $(\kappa,r,b) = (-144.90,1,1)$.  On the right is the bifurcating solution for $\e=-0.01$ where we notice that it resembles a superposition of the two bifurcating linear modes.   }
    \label{fig:resonance}
\end{figure}


\subsection{Example: Bipartite graphs}

For the final numerical exploration we turn to bipartite graphons. Recall from Section~\ref{sec:graphon} that a bipartite graphon is given by
\begin{equation}\label{Bipartite}
	W(x,y) = \begin{cases}
			p & \min\{x,y\} \leq \alpha,\ \max\{x,y\} > \alpha \\
			0 & \mathrm{otherwise}
		\end{cases}
\end{equation}
with parameters $p,\alpha \in (0,1)$. Recall further that if $\alpha \neq \frac 1 2$, then $\mathrm{Deg}(x)$, the degree function defined in \eqref{DegreeFn}, is non-constant and therefore the results of Theorem~\ref{thm:Random} do not apply to the finite node graph Laplacians formed using the graphon $W$. Despite this analytical shortcoming, our numerical results indicate that Theorem~\ref{thm:Random} still applies to bipartite graphs, thus allowing one to study large bipartite graphs using bipartite graphons. Our goal is to illustrate this to the reader throughout this subsection. We begin with the following proposition that details the spectrum of bipartite graphon Laplacians.

\begin{proposition}\label{prop:Bipartite} 
Fix $p,\alpha \in (0,1)$ and let $W$ be a bipartite graphon. Then, the bipartite graphon Laplacian satisfies $\mathcal{L}(W) f = \lambda f$ for some $\lambda \in \mathbb{C}$ and $f \in L^2$, if and only if, one of the following is true:
\begin{enumerate}
	\item $\lambda = 0$ and $f$ is constant,
	\item $\lambda = -p\alpha$, $f(x) = 0$ for all $x \in [0,\alpha)$, and $\int_\alpha^1 f(x)\rmd x = 0$,
	\item $\lambda = -p(1-\alpha)$, $f(x) = 0$ for all $x \in (\alpha,1]$, and $\int_0^\alpha f(x)\rmd x = 0$, 
	\item $\lambda = -p$, there exists a constant $C \in \R$ such that $f(x) = (1-\alpha)C$ for all $x \in [0,\alpha)$, and $f(x) = -\alpha C$ for all $x \in (\alpha,1]$.  
\end{enumerate}
\end{proposition} 

\begin{proof}
	With $W(x,y)$ as given in \eqref{Bipartite}, one finds that 
	\begin{equation}
		[\mathcal{L}(W)f](x) = \begin{cases}
			p\int_\alpha^1 f(y)\rmd y - f(x)p(1-\alpha), & x \leq \alpha, \\
			p\int_0^\alpha f(y)\rmd y - f(x)p\alpha, & x > \alpha,
		\end{cases}
	\end{equation} 
	for all $x \in [0,1]$. The result then follows by setting $\mathcal{L} f = \lambda f$ and working through the distinct cases. 
\end{proof}

\begin{figure} 
    \centering
     \subfigure{\includegraphics[width=0.43\textwidth]{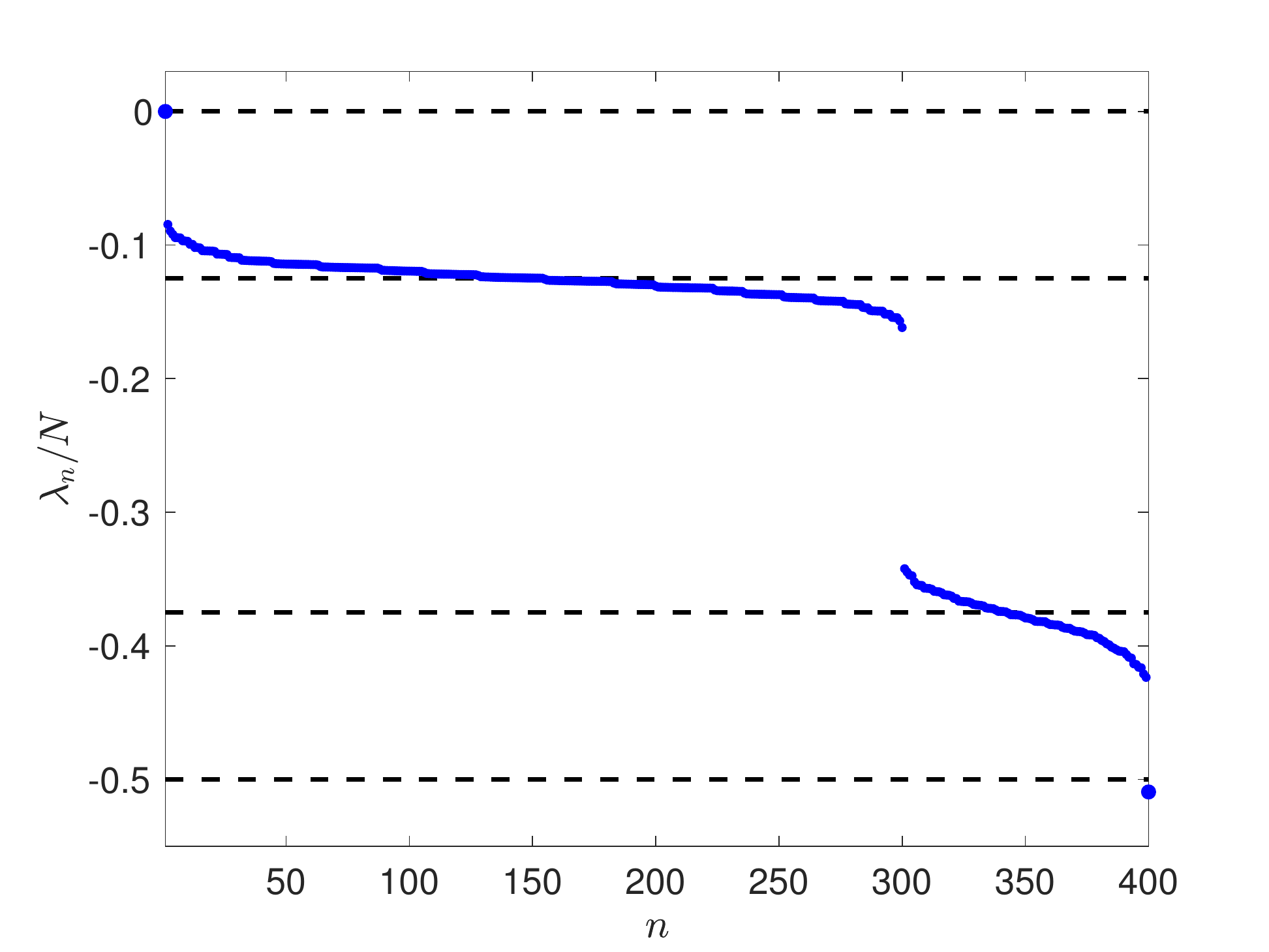}}
 \subfigure{\includegraphics[width=0.43\textwidth]{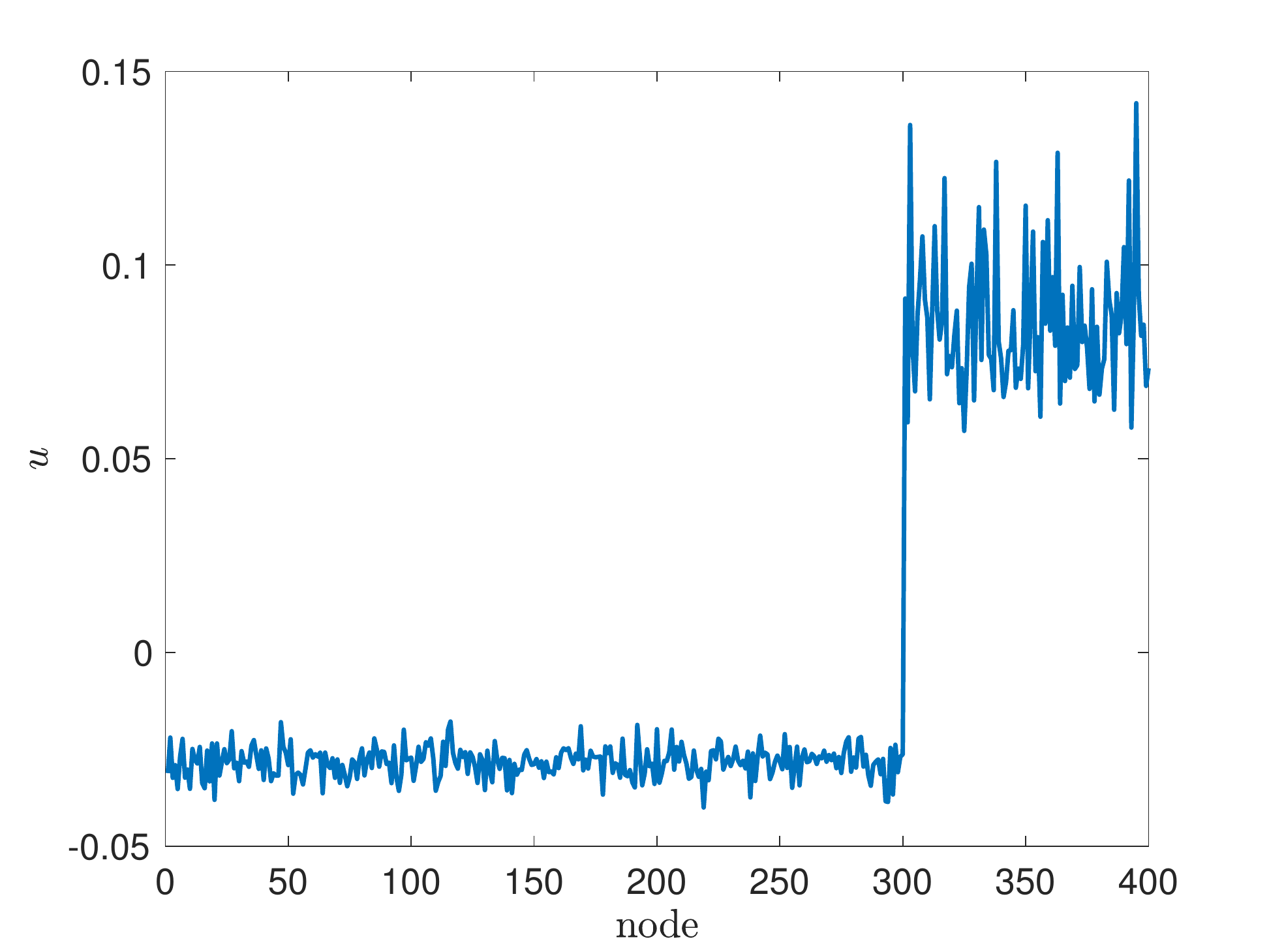}}
   \caption{On the left are the eigenvalues, $\lambda_n$ (blue dots), of a random bipartite graph with $N = 400$ nodes and parameter values $(p,\alpha) = (0.5,0.75)$. Eigenvalues are normalized by the number of nodes so they can be seen to be clustering about the values of eigenvalues of the bipartite graphon Laplacian (dashed horizontal lines) from Proposition~\ref{prop:Bipartite}. On the right is the eigenfunction associated to $\lambda_{400}$, the isolated eigenvalue near $-pN$, which is approximately given by a piecewise constant function, closely resembling the eigenfunction of the bipartite graphon Laplacian that spans the $\lambda = -p$ eigenspace.}
    \label{fig:Bipartite}
\end{figure}

From Proposition~\ref{prop:Bipartite} we see that there are only four possible choices for eigenvalues of the bipartite graphon Laplacian. Moreover, eigenvalues $\lambda = 0,-p$ are isolated with one-dimensional eigenspaces, while the eigenvalues $\lambda = -p\alpha,-p(1-\alpha)$ have infinite-dimensional eigenspaces. This can be observed from the fact that the eigenfunctions can be obtained from any mean-zero function on either $[0,\alpha]$ or $[\alpha,1]$. Such functions can be represented uniquely as Fourier series with zero constant term, meaning that the trigonometric terms that make up the series can be used to span the entire eigenspace. Most important to our work here is that numerically we observe that the random graph Laplacians generated using bipartite graphons have eigenvalues that cluster around the four eigenvalues of the bipartite graphon Laplacian. We illustrate this in Figure~\ref{fig:Bipartite} where we present the eigenvalues of a random bipartite graph Laplacian with $N = 400$ nodes and parameter values $(p,\alpha) = (0.5,0.75)$. Upon normalizing by the size of the graph, one can see the clustering of the random graph eigenvalues near the eigenvalues of the graphon Laplacian. We also provide the eigenfunction of the smallest eigenvalue, $\lambda_{400} \approx -pN = -200$, which resembles a step function that is constant on $[0,\alpha)$ and $(\alpha,1]$, similar to the eigenfunction that spans the one-dimensional eigenspace of $\lambda = -p$ for the graphon Laplacian. We have documented nearly identical results for other choices of the parameters $p$ and $\alpha$. 

With the isolated eigenvalues $\lambda =0,-p$ of the bipartite graphon Laplacian, one may follow the analysis of Sections~\ref{sec:TuringonW} and \ref{sec:Turinggraph} to identify the pattern forming bifurcations for both the continuous-space graphons and the discrete-space random graphs. Interestingly, one finds that the graphon bifurcations are transcritical, as opposed to the pitchfork bifurcations in the ring network case. Since transcritical bifurcations are robust with respect to perturbations, one similarly expects large random bipartite graphs to experience transcritical pattern-forming bifurcations from the eigenvalues near $0$ and $-pN$, where $N$ is the number of nodes in the network. Although not provided here for brevity, our numerical results have confirmed this by observing transcritical bifurcations from these eigenvalues on large random bipartite graphs for a range of parameters $p$ and $\alpha$. 

Finally, we similarly expect that multipartite graphons can be used to study steady-state bifurcations from isolated eigenvalues of large random multipartite graphs with equally as promising results to those that we have provided for ring networks in the preceding sections. We conjecture that our results in Theorem~\ref{thm:Random} hold for multipartite graphons as well, since the only major distinction in the analysis is that multipartite graphons have $\mathrm{Deg}(x)$ being a piecewise constant function, as opposed to the ring network case where they are constant.


\section{Discussion}\label{sec:Discussion}

In this work we have shown how graphons can be used to study pattern-forming Turing bifurcations in spatially-discrete Swift--Hohenberg equations on large random graphs. In particular, the work in Section~\ref{sec:graphon} provides the necessary functional analytical theory to detail how isolated elements of the graphon's spectrum can be used to approximate eigenvalues and eigenvectors of both deterministic and random graphs formed from the graphon. Importantly, the study of graphons allows one to quantify the spectrum of classes of large random graphs, thus providing a novel avenue to study pattern formation on networks which goes far beyond working with regular and idealized graph structures, such as chains and rings. In Section~\ref{sec:TuringonW} we described Turing bifurcations from the trivial state $u = 0$ in the infinite-dimensional graphon Swift--Hohenberg equation. In Section~\ref{sec:Turinggraph} we then leveraged the results of Section~\ref{sec:graphon} to compare the graphon bifurcations with the Turing bifurcations in the Swift--Hohenberg equation posed on large random graphs. We found that the pitchfork bifurcations observed in the graphon setting degenerates into a transcritical bifurcation with a saddle-node nearby when space is discretized to obtain the finite-dimensional graphs.        

Our numerical results in Section~\ref{sec:numerics} provide positive affirmation of our theoretical work in the preceding sections. That is, we were able to observe the degeneration from pitchfork to transcritical bifurcations coming from isolated, non-resonant eigenvalues in the spectrum of small-world random graphs. We also provided an example of a graphon which has resonant eigenvalues, thus necessitating Theorem~\ref{thm:randomresonance}  to describe these Turing bifurcations, again leading to a confirmation of our theoretical results. Our final numerical consideration was that of bipartite graphons, which are not covered by the work in this manuscript. Nonetheless, we found that again the spectrum of the large random bipartite graphs asymptotically cluster about the spectrum of bipartite graphons, as was shown in Figure~\ref{fig:Bipartite}. This indicates that our results are applicable to a wider class of graphons than the ring networks that we restricted ourselves to throughout this work. Hence, in future investigations it will be important to determine the widest possible class of graphons for which our results can apply. Conversely, it is necessary to identify graphons for which our theoretical results herein differ so that follow-up studies can identify their expected pattern-forming bifurcations as well. 

{\color{black} Some investigations of graphons elect to take the underlying spatial domain as an arbitrary probability space $\Omega$, as opposed to the unit interval $[0,1]$. Such generality may not be entirely useful for future investigations though as \cite[Theorem~7.1]{Janson} states that any graphon on a probability space $\Omega$ is equivalent to a graphon on the unit interval $[0,1]$ equipped with the Lesbegue measure. Thus, for many investigations it suffices to restrict oneself to the setting presented in Section~\ref{sec:graphon}. A notable exception to this is when one wishes to understand the effect of a specific spatial structure that is not captured by a one-dimensional interval. For example, one could consider {\em toroidal} graphons on $\Omega = [0,1]^k$, for some $k \geq 2$, which take the form $W(x,y) = T(|x_1 - y_1|, \dots, |x_k - y_k|)$ for each $x = (x_1.\dots,x_k)$ and $y = (y_1,\dots,y_k)$ in $\Omega$ with $T$ being 1-periodic in each of its $k$ arguments. These toroidal graphons generalize the ring graphons with $k = 1$ and their resulting graphs have a multiply-periodic toroidal structure, as opposed to the singly-periodic ring structure studied herein. We have refrained providing such a generalization of our work here since the difficulty of following the added multi-dimensional notation would make the presentation of the results difficult to parse. However, we note that all of our results in this paper can be extended to toroidal graphons by considering hypercubes that discretize $\Omega$ to build the finite vertex graphs and multi-dimensional Fourier series expansions to arrive at the same bifurcation results.

Aside from toroidal graphons, another emerging area for future investigation are those of geometric graphons. These graphons have state space $\Omega = \mathbb{S}^{d-1}$, the $d$-sphere, and take the form $W(x,y) = f(\langle x, y\rangle)$, for some function $f:[-1,1] \to [0,1]$ \cite{GeoGraph1,GeoGraph2}. The resulting finite-vertex random graphs generated from these graphons are known as {\em geometric random graphs}. Such graphs possess the property that they can generate community structure which we expect may lead to localized pattern formation. Thus, the study of geometric graphons presents itself as an important area of future work, with a first step being a complete characterization of their spectrum. }

Even in the context of ring graphons studied in this work, it is important to emphasize that our construction of deterministic and random graphs from graphons results in dense ring networks. That is, the degree of each vertex positively scales with the number of vertices in the graph. Related work motivated by the study of coupled oscillators into semilinear heat equations has extended results from dense graphs to sparse graphs, meaning that the edge density goes to zero as the number of vertices goes to infinity \cite{k17}. Therefore, it is our intention to extend the work in this manuscript to the study of pattern-formation on sparse random graphs, potentially by taking advantage of the recent developments in \cite{borgs19} which details how to use graphons to construct sequences of sparse graphs.  

Finally, we comment on localized patterns on the graph that may be a consequence of the Turing bifurcations detailed in this manuscript. Note that the degeneration of the pitchfork bifurcations from graphon networks to graph networks typically leads to regions of bistability in the spatially-discrete Swift--Hohenberg equation. This can be observed, for example, in Figure~\ref{fig:bifgood} for the small-world network. The competition between states in regions of bistability has been shown to result in steady-state localized patterns for which a connected subset of the elements are near the stable patterned state, while the remaining elements on the graph are near the homogeneous state $u = 0$ \cite{Bramburger1D,BramburgerSquare,TianRing}. It remains to identify if such localized states exist in the small regions of bistability observed in the random graph networks studied here. Since the networks are dense, one expects that the bifurcation curves of localized solutions (if they exist) strongly resemble the simple closed curves of localized patterns on ring networks with all-to-all coupling detailed in \cite{TianRing}. Initial numerical investigations have been unsuccessful in identifying such localized patterns, but without analytical proofs that such steady-state solutions exist or not these investigations are inconclusive. Therefore, we leave the study of localized steady-states on random graphs to a follow-up investigation.

\section*{Acknowledgments} The research of MH was partially supported by the National Science Foundation through NSF-DMS-2007759.  MH is grateful Ben Concepcion for contributing to a preliminary iteration of this project.  {\color{black} The authors thank the anonymous referees whose comments improved the paper.  }

\bibliographystyle{abbrv}
\bibliography{GraphTuring}

\end{document}